\documentclass[a4paper,11pt,reqno]{amsart}
\def\l@subsection{\@tocline{2}{0pt}{30pt}{5pc}{}}
\usepackage{amsfonts,amssymb,amsthm}
\usepackage[mathscr]{eucal}
\usepackage{graphicx}
\usepackage{epsfig}
\usepackage{psfrag}
\usepackage{amsfonts}
\usepackage{amsmath}
\usepackage{amsthm}
\usepackage{latexsym}
\usepackage{verbatim}
\usepackage{mathtools}

  \usepackage[usenames,dvipsnames]{xcolor}

 \usepackage{hyperref}
 
\usepackage[active]{srcltx}      

\usepackage{xcolor,colortbl}
\definecolor{green}{rgb}{0.1,0.1,0.1}
 
\input xy
\xyoption{all}

\theoremstyle{plain}
\newtheorem{thm}{Theorem}

\newtheorem{corollary}[thm]{Corollary}
\newtheorem{lemma}[thm]{Lemma}
\newtheorem{prop}[thm]{Proposition}
\newtheorem{defin}[thm]{Definition}

\newtheoremstyle{exm}
{9pt}{9pt}{}{}{\bfseries}{}{.5em}{}
\theoremstyle{exm}
\newtheorem{exm}[thm]{Example}

\newtheoremstyle{rmk}
{9pt}{9pt}{}{}{\bfseries}{}{.5em}{}
\theoremstyle{rmk}
\newtheorem{rmk}[thm]{Remark}
\newtheoremstyle{question}
{9pt}{9pt}{}{}{\bfseries}{}{.5em}{}
\theoremstyle{question}

\numberwithin{equation}{section}
\numberwithin{thm}{section}
\numberwithin{figure}{section}

\newcommand{\T}{\mathbb{T}}

\newcommand{\M}{\mathsf{M}}
\newcommand{\R}{\mathbb{R}}
\newcommand{\Z}{{\mathbb{Z}}}
\newcommand{\C}{{\mathbb{C}}}

\newcommand{\purge}[1]{}

\title[12, 24 and beyond]{12, 24 and beyond}

\author[L. Godinho, F. von Heymann, S. Sabatini]{Leonor Godinho, Frederik von Heymann and Silvia Sabatini}
\address{Center for Mathematical Analysis, Geometry and Dynamical Systems, Instituto
Superior T\'ecnico, Universidade de Lisboa, Av. Rovisco Pais, 1049-001 Lisboa, Portugal}
\email{lgodin@math.ist.utl.pt}

\address{Mathematisches Institut, Universit\"at zu K\"oln, Weyertal 86-90, D-50931 K\"oln, Germany}
\email{f.vonheymann@uni-koeln.de}

\address{Mathematisches Institut, Universit\"at zu K\"oln, Weyertal 86-90, D-50931 K\"oln, Germany}
\email{sabatini@math.uni-koeln.de}


\date{\today}

\begin{document}

\begin{abstract}
We generalize the well-known ``12" and ``24" Theorems for reflexive polytopes of dimension 2 and 3 to any smooth reflexive polytope. 
Our methods apply to a wider category of objects, here called reflexive GKM graphs, that are associated with certain monotone  
symplectic manifolds which do not necessarily admit a toric action.

As an application, we provide bounds on the Betti numbers for certain monotone Hamiltonian spaces which depend on the \emph{minimal Chern number}
of the manifold. 
\end{abstract}
\maketitle

\tableofcontents

\section{Introduction}\label{intro}

Reflexive polytopes were introduced by Batyrev \cite{Bat} in the context of mirror symmetry. Since then there has been much work to study the interplay between
the combinatorial properties of these polytopes and the geometry of the underlying toric varieties. In particular, the polar dual $\Delta^*$ of a reflexive polytope $\Delta$
is also reflexive, and the pairs $\Delta$ and $\Delta^*$ satisfy a 
surprising combinatorial property in dimensions 2 and 3, involving the relative length of their edges.
We recall that the relative length $l(e)$ of
an edge $e$ of $\Delta$ is the number of its integral points minus 1.

\begin{thm}\label{12 24 comb}
Let $\Delta$ be a reflexive polytope of dimension $n$ with edge set $\Delta[1]$. 
\begin{itemize}
\item[$\bullet$] If $n=2$ then 
\begin{equation}\label{12 general}
\sum_{e\in \Delta[1]} l(e) + \sum_{f\in \Delta^*[1]} l(f)=12\,;
\end{equation}
\item[$\bullet$] If $n=3$ then 
\begin{equation}\label{24 general}
\sum_{e\in \Delta[1]} l(e)l(e^*)=24\,,
\end{equation}
\end{itemize}
where $\Delta^*[1]$ denotes the edge set of the dual polytope $\Delta^*$, and $e^*$ is the edge in $\Delta^*[1]$ dual to $e\in \Delta[1]$.
\end{thm}
Theorem \ref{12 24 comb} has many proofs. 
A non enlightening one is given by exhaustion, since there is only a finite number of reflexive polytopes in each dimension (up to lattice isomorphisms).
However, this method does not add meaning to the statement, in the sense that it does not explain where the 12 and 24 come from. 
Equation \eqref{12 general} has deeper and more intriguing proofs involving, for instance, modular forms, toric geometry and certain relations in $SL_2(\Z)$ (see the beautiful notes \cite{PRV} and 
\cite{HS}). 
The first non-trivial proof of \eqref{24 general} involves toric geometry and was given by Dais, who showed that it is a direct consequence of \cite[Corollary 7.10]{BD}, a result by Batyrev and himself.\footnote{In \cite[Cor.\ 7.10]{BD} the Euler characteristic $e_{\mathrm{st}}(\overline{Z}_f)$ is $24$, since $\overline{Z}_f$ is a K3 surface for $d=3$. Notice that there is a misprint: $(-1)^i$ on the right hand side of the formula should be $(-1)^{i-1}$.} Another purely combinatorial proof
is given in \cite[Section 5.1.2]{HNP}.

There have been several attempts to generalize Theorem \ref{12 24 comb}. In \cite[Sect.\ 9.2]{PRV}, Poonen and Rodriguez-Villegas hint at the possibility of
using the Todd genus of the associated toric variety $M_{\Delta}$ to retrieve combinatorial information on $\Delta$.
In dimensions 2 and 3 the Todd genus is extremely easy to compute, since it
is a combination of the Chern numbers $c_n[M_\Delta]$ and $c_1c_{n-1}[M_\Delta]$. However, in higher dimensions it becomes complicated as
it involves more Chern numbers.  

In this work we use the Chern number $c_1c_{n-1}[M_\Delta]$ to generalize Theorem \ref{12 24 comb} to all Delzant reflexive polytopes, i.e.\ those arising from smooth Fano toric varieties,
since it is exactly the sum of the relative lengths of the edges of $\Delta$.
The key idea behind our results is the existence of a differential equation relating the more general Hirzebruch genus to this Chern number \cite[Theorem 2]{Sal}:
\begin{equation}\label{eq: hirzebruch}
\sum_{e\in \Delta[1]}l(e)=c_1c_{n-1}[M_\Delta]=6 \frac{d^2\chi_y(M_\Delta)}{dy^2}|_{y=-1}+\frac{5n-3n^2}{12}\chi_{-1}(M_\Delta)\,.
\end{equation}
The Hirzebruch genus
is \emph{rigid} for symplectic manifolds admitting Hamiltonian $S^1$-actions \cite[Cor.\ 3.1]{GoSa}, hence in particular for smooth Fano toric varieties. 
This allows us to obtain the following result (see Sect.~\ref{subsec:thmA} and page~\pageref{proof symplectic tools}).

\begin{thm}\label{main combinatorics}
Let $\Delta$ be a Delzant reflexive polytope of dimension $n\geq 2$.  
Denote by $\Delta[1]$ the set of its edges,
and by $l(e)$ the relative length of $e$, for every $e\in \Delta[1]$. 
Let $\mathbf{f}=(f_0,\ldots,f_n)$ be the $f$-vector of $\Delta$. 
Then $\displaystyle\sum_{e\in \Delta[1]}l(e)$ only depends on $\mathbf{f}$. 
More precisely,
\begin{equation}\label{formula f}
\sum_{e\in \Delta[1]} l(e)=12f_2+(5-3n)f_1\,.
\end{equation}
\end{thm}
Expressing \eqref{formula f} in terms of the $h$-vector $\mathbf{h}=(h_0,\ldots,h_n)$ of $\Delta$ we obtain
\begin{equation}\label{def C(n,h)}
\sum_{e\in \Delta[1]} l(e)=C(n,\mathbf{h}):=
\begin{cases}
12\displaystyle\sum_{k=1}^{m} \Big[ k^2h_{m-k}\Big] - m\sum_{k=0}^n h_{k}& \quad\mbox{if }n=2m\mbox{ is even}\\
& \\
12\displaystyle\sum_{k=1}^{m}\Big[ k(k+1)h_{m-k}\Big] - (m-1)\sum_{k=0}^n h_{k} & \quad\mbox{if }n=2m+1\mbox{ is odd.}\\
 \end{cases}
\end{equation}

\begin{rmk}\label{equiv 3}
$\;$\vspace{0.1cm}
\begin{enumerate}
\item Using the Dehn-Sommerville relations for simple polytopes, \eqref{formula f} is equivalent to
\begin{equation}\label{formula f 2}
\sum_{e\in \Delta[1]} l(e)= \frac{24\,f_3}{n-2}+(3-n)\,f_1\,,
\end{equation}
for all $n\geq 3$.

\item Theorem \ref{main combinatorics} admits an immediate generalization to {\bf smooth Gorenstein} polytopes. Indeed, if
$r$ is the index of a Gorenstein polytope $\Delta$, then $r\Delta$ is reflexive and Theorem \ref{main combinatorics} implies that
\begin{equation}\label{gor}
\sum_{e\in \Delta[1]} l(e)=\frac{1}{r}\big(12f_2+(5-3n)f_1\big)\,.
\end{equation}

\end{enumerate}
\end{rmk}
We provide two alternative simpler proofs of Theorem \ref{main combinatorics} that do not involve the Hirzebruch genus. One is entirely combinatorial (see Sect.~\ref{rp} and page~\pageref{page:combinatorics}), and the other  uses symplectic toric geometry (see Sect.~\ref{mtm} and page~\pageref{proof symplectic}).

\begin{rmk}\label{non smooth}
The non-smooth case builds on our results, and is the subject of a forthcoming paper \cite{GHS2}.
\end{rmk}

Theorem \ref{main combinatorics} is a very special case of a much more general phenomenon that does not involve toric geometry, but only a much smaller symmetry.  
From the Delzant Theorem \cite{De}, to every $n$-dimensional smooth--or Delzant--polytope one can associate a compact symplectic manifold of dimension $2n$ endowed with a Hamiltonian action of a torus 
of dimension $n$, called {\bf symplectic toric manifold} (see Sect.\ \ref{sts}). When the torus acting is just a circle, and the fixed points are isolated, the symplectic manifold $(M,\omega)$ together with the moment map
$\psi\colon (M,\omega)\to \R$ is called a {\bf Hamiltonian $S^1$-space}. This category includes that of {\bf Hamiltonian GKM spaces}, introduced in the seminal paper \cite{GKM}, which also plays an important role in this paper. 
Many of these Hamiltonian $S^1$-spaces posses special sets of smoothly embedded $2$-spheres, called {\bf toric 1-skeletons}. Geometrically, the class of a toric 1-skeleton in $H_2(M;\Z)$ is  Poincar\'e Dual to $c_{n-1}$ (see Lemma~\ref{P Dual}). When the manifold is symplectic toric with moment polytope $\Delta$,
the 1-skeleton is unique, and
corresponds to the pre-image of the edges of $\Delta$ by the moment map. A similar statement is true for Hamiltonian GKM-spaces (see Sect.\ \ref{gt3}). Note that there are no examples known of Hamiltonian $S^1$-spaces that do not admit a toric $1$-skeleton.

If $\Delta$ is the Delzant polytope associated to a symplectic toric manifold $(M,\omega,\psi)$, the $h$-vector of $\Delta$ corresponds to the vector $\mathbf{b}:=(b_0,\ldots,b_{2n})$
of even Betti numbers 
of $M$. Then Theorem \ref{main combinatorics} is a consequence of a more general result,
where we define
$C(n,\mathbf{b})$ to be $C(n,\mathbf{h})$, with $h_j$ replaced by $b_{2j}$, for all $j=0,\ldots,n$, namely

\begin{equation}\label{def C(n,b)}
C(n,\mathbf{b}):=
\begin{cases}
12\displaystyle\sum_{k=1}^{\frac{n}{2}} \Big[ k^2b_{n-2k}(M)\Big]- \frac{n}{2}\chi(M)& \quad\mbox{if }n \mbox{ is even}\\
& \\
12\displaystyle\sum_{k=1}^{\frac{n-1}{2}}\Big[ k(k+1)b_{n-1-2k}(M)\Big] -  \left(\frac{n-3}{2}\right)\chi(M)& \quad\mbox{if }n \mbox{ is odd.}\\
\end{cases}
\end{equation}

\begin{thm}\label{A}
Let $(M,\omega,\psi)$ be a Hamiltonian $S^1$-space of dimension $2n$, and let $\mathbf{b}:=(b_0,\ldots,b_{2n})$
be the vector of its even Betti numbers. 
Let $c_1\in H^2(M;\Z)$ be the first Chern class of the tangent bundle. 
If $(M,\omega,\psi)$ admits a toric 1-skeleton $\mathcal{S}=\cup_{e\in E}\{S_e^2\}$, then
the sum of the integrals of $c_1$ on the spheres corresponding to the toric 1-skeleton only depends on the topology of $M$. More precisely 
\begin{equation}\label{intro:formula}
\sum_{e\in E} c_1[S^2_e]=C(n,\mathbf{b})\,.
\end{equation}
In particular, if $(M,\omega,\psi)$ is a symplectic toric manifold with moment polytope $\Delta$, and $\mathcal{S}=\cup_{e\in \Delta[1]}S_e^2$ is the set of spheres in one-to-one correspondence with the edges
 of the moment polytope $\Delta$, then 
\begin{equation}\label{intro:formula2}
\sum_{e\in E} c_1[S^2_e]=12f_2+(5-3n)f_1 \,,
\end{equation}
where $\mathbf{f}=(f_0,\ldots,f_n)$ is the $f$-vector of $\Delta$.
\end{thm}
The proof of Theorem \ref{A} is given in Section \ref{gt3}. For the particular case of \eqref{intro:formula2} see Theorem \ref{formula c1}.

Inspired by the Mukai conjecture \cite{Muk}, in Sect. \ref{mhs} we apply this theorem to {\bf monotone Hamiltonian $S^1$-spaces} $(M,\omega,\psi)$, namely those for which
$c_1=r[\omega]$, for some $r>0$,  obtaining restrictions
for the Betti numbers of $M$ depending on the {\bf index} $k_0$ of $(M,\omega)$. This is defined as the largest integer $k$ such that $c_1=k \,\eta$, for some non-zero $\eta\in H^2(M;\Z)$.
In \cite[Cor.\ 1.3]{Sa} it is shown that, for Hamiltonian $S^1$-spaces, one has $1\leq k_0\leq n+1$, the same bound that holds for Fano manifolds \cite[Cor.\ 7.17]{Mi}. Here we prove that for every $k_0\geq n-2$, if the sequence of even Betti numbers is unimodal\footnote{A sequence
$(b_0,\ldots,b_N)$ is called unimodal if $b_i\leq b_{i+1}$ for all $i\leq \frac{N}{2}-1$.}, then there are finitely
many possibilities for the Betti numbers of $M$. In particular, if $k_0=n+1$, then $b_{2j}=1$ for all $j=0,\ldots,n$. 
If $k_0=n$ then $b_{2j}=1$ for all $j=0,\ldots,n$ for $n$ odd, and, if  $n$ is even, $b_{2j}=1$ for all $j\in \{0,\ldots,n\}\setminus \{\frac{n}{2}\}$ and $b_n=2$
(see Corollaries \ref{cor 2}, \ref{cor 3.1} and \ref{cor 3}). For smooth toric Fano manifolds, we obtain restrictions for the possible 
 $h$-vectors (see Corollaries \ref{inequalities h f} and \ref{cor k0}). The results in Corollary \ref{cor k0} are not new, since Fano varieties of large index are completely classified.
 Nevertheless, our methods are different and do not involve any of the algebraic/toric geometric tools usually used in this classification.

In Sect.\ \ref{mg} we generalize the concept of a reflexive (Delzant) polytope to that of a {\bf reflexive (GKM) graph}, which shares many of its properties, and we prove
the analogue of Theorem \ref{main combinatorics} for these objects (see Corollary \ref{main combinatorics 2}). Finally, in Sect.\ \ref{flags} we exhibit a class of
reflexive GKM graphs, namely those associated with coadjoint orbits endowed with a monotone symplectic structure.

\begin{figure}[htbp]
\begin{center}
\includegraphics[width=12cm]{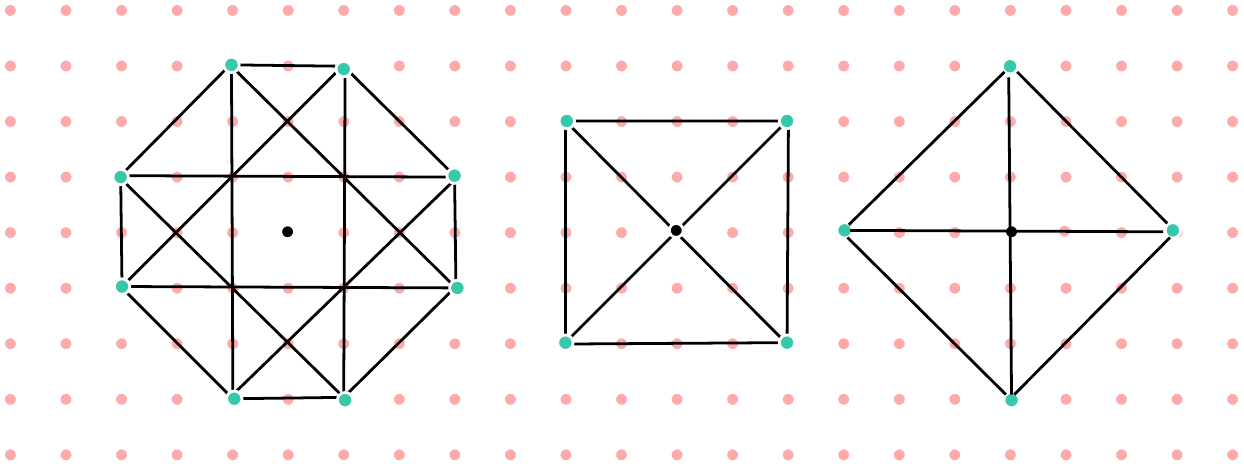}
\caption{Examples of reflexive GKM graphs, see Example \ref{a2b2}.}
\label{monotone graphs}
\end{center}
\end{figure}

\textbf{Acknowledgements} We would like to thank
Dimitrios Dais for explaining the proof of \eqref{24 general} to us, and Michael Lennox Wong for helping us with the proof of Proposition \ref{graph flag mon}. 
Moreover, we are grateful to Klaus Altman, Christian Haase, Lutz Hille, Benjamin Nill, Milena Pabiniak, Sinai Robins, Thomas Rot, J\"org Sch\"urmann, and Kristin Shaw for fruitful discussions.

\section{Theorem \ref{main combinatorics}: the combinatorial proof}
\subsection{Delzant polytopes}
Consider $\R^n$ with the standard scalar product $\langle \cdot, \cdot  \rangle$ and lattice $\Z^n\subset \R^n$. 
We recall that an integral vector $w\in \Z^n$ is called {\bf primitive} in the lattice $\Z^n$ if $w=m\cdot \widetilde{w}$, for some $\widetilde{w}\in \Z^n$ and $m\in \Z$,  implies $m=\pm 1$. 

Let $\Delta$ be an $n$-dimensional polytope in $\R^n$. Then $\Delta$
admits a (unique) minimal representation as an intersection of half spaces
\begin{equation}\label{delta half}
\Delta=\bigcap_{i=1}^k\{x\in \R^n \mid \langle x , l_i \rangle \leq m_i\}\,.
\end{equation}
 The hyperplanes $H_i=\{x\in \R^n \mid \langle x , l_i \rangle = m_i\}$ are exactly those supporting the $(n-1)$-dimensional faces of $\Delta$, called \textbf{facets}.  
 If $\Delta$ is \textbf{integral}, namely all the vertices
belong to $\Z^n$, then
 each $l_i$ in \eqref{delta half}
can be chosen to be in $\Z^n$, or more precisely it can be chosen to be the (unique) primitive \emph{outward normal} vector to $H_i$. With such a choice,
the $m_i$'s in the equations above are uniquely determined.  
\begin{defin}\label{delzant def}
Let $\Delta$ be an $n$-dimensional polytope in $\R^n$, and consider the lattice $\Z^n\subset \R^n$. Such a polytope is called {\bf Delzant} if:
\begin{itemize}
\item[(D1)] it is \emph{simple}, i.e.  there are exactly $n$ edges meeting at each vertex;
\item[(D2)] it is \emph{rational}, i.e. the edges meeting at each vertex $v$ are of the form $v+t w_i$ with $t\geq 0$ and $w_i\in \Z^n$;
\item[(D3)] it is \emph{smooth} at each vertex, i.e. for every vertex  the corresponding  $w_1,\ldots,w_n\in \Z^n$ defined in \emph{(D2)} can be chosen to be a basis of $\Z^n$ (i.e. $\Z\langle w_1,\ldots,w_n \rangle = \Z^n$).
\end{itemize}
\end{defin}
Henceforth, the vectors $w_1,\ldots,w_n$ defining the directions of the edges at $v$ are always chosen so that (D2) and (D3) are satisfied, and are called the {\bf weights} at the vertex $v$. 

\begin{rmk}\label{connections}
Delzant polytopes are also known in  combinatorics and toric geometry literature as \emph{smooth} or \emph{unimodular polytopes} (see \cite[Def.\ 2.4.2 (b)]{Cox})
\end{rmk}
\begin{defin}\label{hv}
Let $\Delta$ be an $n$-dimensional simple polytope. 
\begin{itemize}
\item[$\bullet$]  We denote the set of its $i$-dimensional faces  by $\Delta[i]$, for all $i=0,\ldots,n$.
\item[$\bullet$] The $f$-vector of $\Delta$ is given by $\mathbf{f}=(f_0,\ldots,f_n)$, where $f_i$ is the cardinality of $\Delta[i]$;
\item[$\bullet$] The $h$-vector of $\Delta$ is given by $\mathbf{h}=(h_0,\ldots,h_n)$, where
$$
h_j=\sum_{i=0}^j (-1)^{j-i}\binom{n-i}{n-j}f_{n-i}\quad\mbox{for all}\quad j=0,\ldots,n\,.
$$
\end{itemize}
\end{defin}
Some authors define the $h$-vector for simplicial polytopes, i.e., the duals of simple polytopes (see, e.g., \cite{Zie, Ad}). Note that dualization transforms $f_i$ into $f_{n-i-1}$. 

\begin{prop}\label{int_dist}
Let $\Delta$ be an $n$-dimensional Delzant polytope, and let $u,v\in\Delta[0]$ be vertices that are connected by an edge $e$ in $\Delta[1]$. 
Let $\{w_1,\ldots,w_n\}$ and $\{\tilde{w}_1,\ldots,\tilde{w}_n\}$ be the weights at $u$ and $v$, respectively. We may assume that $w_1$ and $\tilde{w}_1$ point along $e$, hence
$\tilde{w}_1=-w_1$.

Then, for every $i\in\{2,\ldots,n\}$, there exists $j_i\in\{2,\ldots,n\}$ such that 
\[
w_i-\tilde{w}_{j_i}=a_iw_1
\]
for some $a_i \in \Z$. 
\end{prop}
It will be clear from the proof that each $a_i$ is uniquely determined by  the edge $e$ and by the $2$-dimensional face $F_i$
contained in the affine space 
 $u+\R\langle w_1,w_{i}\rangle=v+\R\langle \tilde{w}_1,\tilde{w}_{j_i}\rangle$, 
 for all $i=2,\ldots,n$, where $\R\langle w,\tilde{w}\rangle$ denotes the $\R$-linear span of $w$ and $\tilde{w}$.
 Henceforth we denote $a_i$ by $a_i^e$, and call it the {\bf normal contribution of} $e$ {\bf in (the face)} $F_i$ (containing $e$).

\begin{proof}
Let $F_i$ be the 2-dimensional face of $\Delta$ that is contained in $u+\R\langle w_1,w_{i}\rangle$. 
Consider the edge $e_{j_i}\in \Delta[1]$ starting at $v$ and belonging to $F_i$, and let
$\tilde{w}_{j_i}$ be the weight at $v$ along $e_{j_i}$. Note  that
 $\tilde{w}_{j_i}\in \R\langle w_1,w_{i}\rangle$.

By condition (D3) in Definition \ref{delzant def}, $\{w_1,w_i\}$ can be extended to a lattice basis of $\Z^n$. Hence there exists $\alpha,\beta\in\Z$ such that $\alpha w_1+\beta w_i= \tilde{w}_{j_i}$.
 Since $\tilde{w}_1=-w_1$ and $\{\tilde{w}_1,\tilde{w}_{j_i}\}$ can also be extended to a lattice basis of $\Z^n$, we deduce that $\beta=\pm 1$.
As $\beta=-1$ would imply that $e$ has points in the interior of $F_i$, the claim follows.
\end{proof}

The next theorem is one of the key ingredients to prove Theorem \ref{main combinatorics}. 
\begin{thm}\label{combinatorics 2}
Let $\Delta$ be a Delzant polytope of dimension $n\geq 2$ with $f$-vector $\mathbf{f}=(f_0,\ldots,f_n)$. Then
\begin{equation}\label{sum aie}
\sum_{e\in \Delta[1]} \sum_{i=2}^n a_i^e =12f_2-3 (n-1) f_1
\end{equation}
where the $a_i^e$'s are the normal contributions to the edge $e$.
\end{thm}

Theorem \ref{combinatorics 2} is a generalization to every dimension of the following well-known fact in dimension $2$, whose proof is inspired by toric geometry, but can be
made entirely combinatorial (see \cite[Result 57]{Oda}, \cite[pages 43-44]{Ful}, and in a similar fashion \cite[Theorem 5.1.9]{HNP}).

\begin{prop}\label{dim2}
If $\Delta$ is a $2$-dimensional Delzant polytope, then
\[
\sum_{e\in \Delta[1]} a^e = 12-3|\Delta[0]|,
\] 
where $a^e$ is the (only) normal contribution of $e$ (in $\Delta$), as defined in Proposition~\ref{int_dist}.
\end{prop}
 The idea to prove Theorem \ref{combinatorics 2} is to use Proposition \ref{dim2} for each $2$-dimensional face of $\Delta$. 
In order to do so, we first need to prove the following result.
\begin{lemma}\label{2face}
Given a Delzant polytope $\Delta$ of dimension $n\geq 2$, each face $F\in \Delta[2]$ is a Delzant polytope of dimension $2$ with respect to a lattice $\ell_F\subseteq \Z^n$.
\end{lemma}
\begin{proof}
For each vertex $u$ in $F$, let $w_1^u$ and $w_2^u$ be the two weights at $u$ pointing along the edges of $F$ starting at $u$. 
For each such vertex, define the two-dimensional lattice
$\ell^u:=\Z\langle w_1^u,w_2^u\rangle$. We claim that $\ell^u$ is independent of $u\in F[0]$, the set of vertices of $F$, and call this lattice $\ell_F$.
This follows easily from the fact that, by  (D3) in Definition \ref{delzant def}, at each vertex $u$ the set $\{w_1^u,w_2^u\}$ extends to a $\Z$-basis of $\Z^n$,
and the linear span $\R\langle w_1^u,w_2^u\rangle$ does not change with $u\in F[0]$. 
Hence $F$ is a Delzant polytope in $\R\langle w_1^u,w_2^u\rangle$ w.r.t.\ the lattice $\ell_F$ (cf.\ Definition \ref{delzant def}, where one needs to replace $\R^n$ with $\R\langle w_1^u,w_2^u\rangle$
and the lattice $\Z^n$ with $\ell_F$). 
\end{proof}
Clearly the same proof can be adapted to prove that every face $F\in \Delta[k]$ is a Delzant polytope w.r.t. a  lattice $\ell_F \subseteq \Z^n$, for all $k=0,\ldots,n$.
\begin{proof}[Proof of Theorem \ref{combinatorics 2}]
Let $F\in \Delta[2]$. Define $\Phi\colon \R\langle w_1^u,w_2^u\rangle\to \R^2$ to be a linear map
that brings $\ell_F$ to $\Z^2$, and, more precisely, a $\Z$-basis of $\ell_F$ to the standard $\Z$-basis of $\Z^2$. 
Under this transformation, the `translated face' $-u+F\subset \R^n$, for $u\in F[0]$, is mapped to a Delzant polygon $\widetilde{F}$ in $\R^2$, endowed with lattice $\Z^2$.
Also, for each $u\in F[0]$ and $e\in F[1]$ contained in $\{u+tw_1^u,\,t\geq 0\}$, the 
normal contribution $a^e_F$ of $e$ in $F$ is the same as the normal contribution of the edge of $\widetilde{F}$ contained in $\{t\Phi(w_1^u),\,t\geq 0\}$.
Hence Proposition \ref{dim2} implies that
\begin{equation}\label{sum1}
\sum_{e\in F[1]} a^e_F=12-3|F[0]|\,.
\end{equation}
Moreover, it is easy to check that
\begin{equation}\label{aes}
\sum_{e\in \Delta[1]} \sum_{i=2}^n a_i^e = \sum_{F\in\Delta[2]}\sum_{e\in F[1]} a^e_F\,.
\end{equation}
From \eqref{sum1} we can conclude that
$$
\sum_{e\in \Delta[1]} \sum_{i=2}^n a_i^e = \sum_{F\in\Delta[2]}\sum_{e\in F[1]} a^e_F= 12 f_2-3{n \choose 2} f_0\,,
$$
where ${n \choose 2}$ is precisely the number of $2$-dimensional faces containing a given vertex (here we used that $\Delta$ is a simple polytope, see (D1) in Definition \ref{delzant def}). The conclusion follows from observing that $2f_1=nf_0$.
\end{proof}

\subsection{Reflexive polytopes} \label{rp}
First we recall the definition of a reflexive polytope, as it was first
introduced by Batyrev in \cite{Bat}.
\begin{defin}\label{reflexive}
Let $\Delta$ be an $n$-dimensional polytope. Then $\Delta$ is called {\bf reflexive} if it is an integral polytope in $\R^n$ containing $\mathbf{0}$ in its interior such that
\begin{equation}\label{refl eq}
\Delta=\bigcap_{i=1}^k\{x\in \R^n \mid \langle x , l_i \rangle \leq 1\}\,,
\end{equation}
where the $l_i\in \Z^n$ are the primitive outward normal vectors to the hyperplanes $H_i$ defining the facets, for $i=1,\ldots,k$.
\end{defin}
Reflexive polytopes have many properties. For instance, from \eqref{refl eq} it is easy to see that $\mathbf{0}$ must be the only interior integral point.
Another important property is related to the dual polytope $\Delta^*$, defined as 
\begin{equation}\label{dual}
\Delta^*=\{y\in \R^n \mid \langle x , y \rangle \geq -1 \;\;\;\mbox{for all}\;\;\; x\in \Delta\}\,.
\end{equation} 
Indeed, a lattice polytope $\Delta$ containing $\mathbf{0}$  is reflexive if and only if $\Delta^*$ is an integral polytope; more is true: $\Delta$ is reflexive if and only if $\Delta^*$ is reflexive (see \cite[Theorem 4.1.6]{Bat}). Moreover, since  $\mathbf{0}$ is in the interior of $\Delta$, we have $\Delta^{**}=\Delta$.
Finally, since every reflexive polytope contains only one interior integer point, a result of Lagarias and Ziegler \cite{LZ} implies that, up to lattice isomorphism,
there is only a finite number of reflexive polytopes  in each dimension.

There exists a duality between faces of $\Delta$ of codimension $k$ and faces of $\Delta^*$ of dimension $k-1$. Indeed every face $F$
of $\Delta$ of codimension $k$ can be written as $
F=\displaystyle\cap_{l=1}^k H_{j_l}\cap \Delta
$,
where $H_{j_l}=\{x\in \R^n \mid \langle x , l_{j_l} \rangle = 1\}$ is a  hyperplane supporting one of the facets of $\Delta$, for $l=1,\ldots,k$.  
Then $F^*$ is defined to be the convex hull of the points $-l_{j_1},\ldots,-l_{j_k}$ of $\Delta^*$.
For instance, for $n=2$ the dual of a vertex $v$ in $\Delta$ is an edge $e^*$ in $\Delta^*$, and for $n=3$ the dual of an edge $e$ in $\Delta$ is an edge $e^*$ in $\Delta^*$. 
\begin{defin}\label{int length} 
Let $e=(v_1,v_2)$ be a segment in $\R^n$ from $v_1$ to $v_2$ such that 
$$v_2-v_1= l(e) w$$ for some primitive $w\in \Z^n$ and $l(e)\in \R^+$. 
Then $l(e)$ is called the  {\bf relative length} of $e$.
\end{defin}
For lattice polytopes, the relative length is well-defined for each of their edges. In particular this holds for
reflexive polytopes.
In dimensions $2$ and $3$, the relative lengths of the edges of $\Delta$ and those of the dual are related by the
striking formulas of Theorem \ref{12 24 comb}. 

One of the main goals of this section is to provide an entirely combinatorial proof of Theorem \ref{main combinatorics}, which is a generalization of Theorem \ref{12 24 comb} to every dimension for \emph{Delzant} reflexive polytopes (see Definition \ref{delzant def}). Theorem \ref{main combinatorics} has two additional proofs: the first is the translation of the combinatorial proof to toric geometry (it was indeed the `toric geometry proof' that inspired
the combinatorial one); the second involves the Hirzebruch genus of a Hamiltonian space, and allows us to generalize Theorem \ref{main combinatorics} to the much broader category of objects called \emph{reflexive GKM graphs} (see Sect.~\ref{subsec:thmA} and \ref{mg}).

The next proposition is not new (see \cite[Prop.\ 1.8]{EP} and \cite[Sect.\ 3]{M}). However, since this note is aimed at readers coming from different backgrounds, we include a proof
for the sake of clarity and completeness. 

\begin{prop}\label{equivalent}
Let $\Delta\subset \R^n$ be an $n$-dimensional Delzant polytope with $\mathbf{0}$ in its interior. Then the following conditions are equivalent:
\begin{itemize}
\item[{\bf (i)}] $\Delta$ is a reflexive polytope;
\item[{\bf (ii)}] For every vertex $v\in \Delta$ we have 
\begin{equation}\label{vertex fano}
\sum_{j=1}^n w_j = -v
\end{equation}
where $w_1,\ldots,w_n$ are the weights at $v$.
\end{itemize}
\end{prop}
\begin{rmk}
Condition {\bf (ii)} above corresponds exactly to the \emph{vertex-Fano} condition defined by McDuff in \cite{M} for Delzant polytopes (see \cite[Def.\ 3.1]{M}): here such polytopes
are called \emph{monotone}. This condition is the key idea to generalize the concept of reflexive polytope to that of reflexive GKM graph (see Definition \ref{def monotone}).
\end{rmk}
\begin{proof}[Proof of Prop.\ \ref{equivalent}]
{\bf (i)}$\implies${\bf (ii)} Assume that $\Delta$ is a Delzant reflexive polytope containing the origin $\mathbf{0}$ in its interior. Let $v$ be a vertex of $\Delta$, and consider the hyperplanes 
$H_i=\{x\in \R^n \mid \langle x, l_i \rangle=1\}$ supporting the facets of $\Delta$ containing $v$, where $l_i$ is the outward primitive integral
vector orthogonal to $H_i$. 
Since, by assumption, $\Delta$ is Delzant, property (D3) implies that $\mathbf{0}$
can be written as $\mathbf{0}=v+k_1\,w_1+\cdots + k_n\,w_n$, for some unique $n$-tuple of integers $k_1,\ldots,k_n$.
Pick one of the hyperplanes $H_i$ containing $v$, and suppose that the vectors $w_1,\ldots,\widehat{w_i},\ldots, w_n$ are tangent to $H_i$.
Then 
$$
0=\langle v+k_1w_1+\cdots + k_nw_n,l_i\rangle = \langle v,l_i \rangle + k_i \langle w_i,l_i \rangle = 1+k_i \langle w_i,l_i \rangle\,.
$$
Observe that $k_i$ and $\langle w_i,l_i \rangle$ are both integers, and that their product is $-1$. From the choice of 
$l_i$ and $w_i$ it follows that $\langle w_i,l_i \rangle=-1$, implying $k_i=1$. Since the above argument holds for every $i=1,\ldots,n$
and for every vertex $v\in \Delta$, {\bf (ii)} follows. \\
{\bf (ii)}$\implies${\bf (i)} Let $\Delta$ be a Delzant polytope of dimension $n$ containing the origin in its interior, let $v$ be one of its vertices and $w_1,\ldots,w_n$ as in {\bf (ii)}. From 
the smoothness of $\Delta$ it follows that there exists a lattice transformation taking the vectors $w_1,\ldots,w_n$ to the standard 
vectors $e_1=(1,0,\ldots,0),\ldots, e_n=(0,\ldots,0,1)$, and so the hyperplanes containing $v$ become $H_i'=\{x_i=h_i\}$, for some
$h_i\in \Z$, for every $i=1,\ldots,n$. However, \eqref{vertex fano} forces all the $h_i$'s to be $-1$, implying that $\Delta$ is reflexive.
\end{proof}
We introduce the following definition.
\begin{defin}\label{cones}
Let $\Delta$ be a rational polytope, i.e. the edges meeting at each vertex $v$ are of the form $v+t w_i$ with $t\geq 0$ and $w_i\in \Z^n$, for all $i=1,\ldots,k(v)$ (here $k(v)$
denotes the number of edges incident to $v$). 
\begin{itemize}
\item[(i)] The {\bf cone at $v$} is defined to be $$\mathcal{C}_v:=\{\sum_{j=1}^{k(v)} t_j w_j\mid t_j \geq 0\}.$$
\item[(ii)] The {\bf tangent cone (or vertex cone) at $v$} is the affine cone $\mathcal{C}_v^{\text{aff}}$ given by $v+ \mathcal{C}_v$.
\end{itemize}
\end{defin}
The above concepts are defined for rational polytopes, but can be easily generalized to non-rational ones. 
\begin{rmk}\label{cones equiv}
It is clear that the collection of tangent cones at the vertices determines the polytope itself (indeed, the collection of vertices does), but the collection of cones in general does not.
However, Proposition \ref{equivalent} (ii) implies that, \emph{for Delzant reflexive polytopes knowing the cone at a vertex $v$ is equivalent to knowing its tangent cone}.
Hence \emph{every Delzant reflexive polytope is determined by the collection of cones at its vertices}.
\end{rmk}

The next proposition is the second key ingredient to prove Theorem \ref{main combinatorics}.
\begin{prop}\label{length_Delzant}
Let $\Delta$ be an $n$-dimensional Delzant reflexive polytope, and $l(e)$ the relative length of $e\in \Delta[1]$.
Then, for each $e\in \Delta[1]$, we have
\begin{equation}
 l(e) = 2 + \sum_{i=2}^n a_i^e,
\end{equation}
where the $a_i^e$'s are the normal contributions to the edge $e$, as in Proposition \ref{int_dist}.
\end{prop}
\begin{proof}
Let $e\in \Delta[1]$ with endpoints $u$ and $v$. 
Assume that  $\{w_1,\ldots,w_n\}$ and $\{\tilde{w}_1,\ldots,\tilde{w}_n\}$ are the weights at $u$ and $v$, respectively, and that $w_e:=w_1=-\tilde{w}_1$. 

Observe that $v-u=l(e) \cdot w_e$. On the other hand, since $\Delta$ is reflexive, Proposition \ref{equivalent} implies that
\[
v-u=\sum_{i=1}^n w_i -\sum_{i=1}^n \tilde{w}_i= \left( 2+\sum_{i=2}^n a_i^e\right)\cdot w_e.
\] 
\end{proof}
Proposition \ref{equivalent} is crucial in the proof of Proposition \ref{length_Delzant}: for Delzant reflexive polytopes it allows us to 
turn a `metric quantity' (the relative length $l(e)$) into an `intrinsic' property of the polytope, namely the sum of the normal contributions.

\begin{proof}[Proof of Theorem \ref{main combinatorics}]\emph{(Combinatorial)}\label{page:combinatorics}
By Proposition \ref{length_Delzant} and Theorem \ref{combinatorics 2} we have that
$$
\sum_{e\in \Delta[1]}l(e)=\sum_{e\in \Delta[1]}\Big( 2+ \sum_{i=2}^n a_i^e\Big)=12f_2+(5-3n)f_1\,.
$$
To pass from \eqref{formula f} to \eqref{def C(n,h)} it is sufficient to express the $f$-vector in terms of the $h$-vector by the formula
$
f_k=\sum_{l=k}^n {l \choose k}h_{n-l}
$, which holds for all $k=0,\dots,n$. The details are left to the reader. 
 \end{proof}

\begin{corollary}\label{cor combinatorics}
As special cases we have:
\begin{itemize}
\item[$\bullet$] If $\Delta$ is a Delzant reflexive polytope of dimension $2$ then
\begin{equation}\label{12}
\sum_{e\in \Delta[1]}l(e)+|\Delta[0]|= 12. 
\end{equation}
\item[$\bullet$] If $\Delta$ is a Delzant reflexive polytope of dimension $3$ then 
\begin{equation}\label{24}
\sum_{e\in \Delta[1]}l(e)= 24 \,.
\end{equation}
\end{itemize}
\end{corollary}

\begin{figure}[h]
    \centering
    \includegraphics[width=10cm]{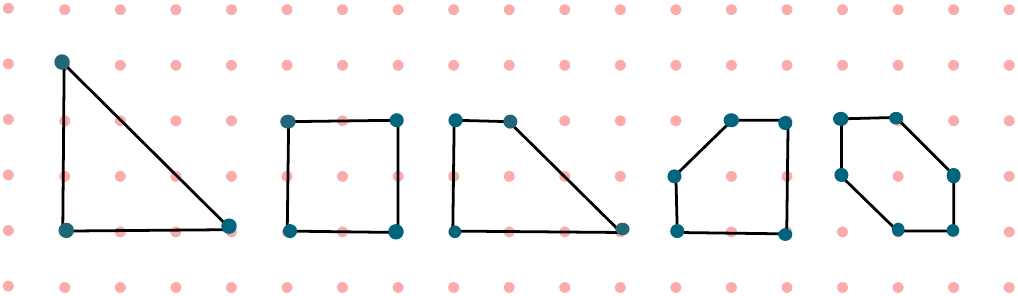}
    \caption{Delzant reflexive polygons}
    \label{fig:smooth}
\end{figure}

\begin{rmk}\label{rmk comb}
Equations \eqref{12} and \eqref{24} are special cases of \eqref{12 general} and \eqref{24 general}, since the smoothness of $\Delta$
implies that $l(f)=1$ for all $f\in E^*$. Indeed, for $n=2$ it is easy to see that, if $v$
is a vertex of $\Delta$ and $w_1,w_2$ are the weights at $v$, then $|\det(w_1,w_2)|=l(e^*)$, where
$e^*$ is the edge in $\Delta^*$ dual to $v$. Hence, for Delzant polytopes of dimension $2$, we have $\sum_{f\in E^*}l(f)=|V|$.
Analogously, for $n=3$ we have $l(f)=1$ for all $f\in E^*$. Indeed, let $H_1$ and $H_2$ be the two hyperplanes supporting the facets $F_1$ and $F_2$ of $\Delta$ intersecting in the edge $e$, and let $l_1$ and $l_2$ be the two primitive outward normal vectors to $F_1$ and $F_2$. From the smoothness of $\Delta$, there exists a $GL(3,\Z)$ transformation that sends a neighborhood of $e$ in  $\Delta$ into a cone in $\R^3$ with apex generated by the vector $(0,0,1)$, with $l_1$ and $l_2$ becoming the vectors $(-1,0,0)$ and $(0,-1,0)$. Then the edge dual to $e$ becomes the segment connecting $(1,0,0)$ to $(0,1,0)$, which has relative length $1$. Since relative lengths are invariant under $GL(3,\Z)$ transformations, it follows that $l(f)=1$ for all $f\in E^*$.
\end{rmk}

\section{Theorem \ref{main combinatorics}: the (symplectic) toric proof}\label{sts}

\subsection{Preliminaries: Hamiltonian $\T$-spaces}\label{hts}
Let $\T$ be a compact real torus of dimension $d$ with integral lattice  $\ell\subset Lie(\T)$ acting effectively  on a compact symplectic manifold $M$ with a  {\bf discrete fixed point set} $M^\T$. Assume that the action is Hamiltonian, i.e.  that there exists a smooth $\T$-invariant map $\psi\colon M\to Lie(\T)^*$ such that
\begin{equation}\label{mm}
d\langle \psi(\cdot), \xi \rangle=-\iota_{\xi^\#}\omega,
\end{equation} 
where $\xi^\#$ is the vector field associated to $\xi\in Lie (\T)$. (Here $\langle \cdot, \cdot \rangle$ denotes the pairing between $Lie(\T)^*$ and $Lie(\T)$.)  
We call a triple $(M,\omega,\psi)$ with these properties  a \textbf{Hamiltonian $\T$-space}\label{hamtspace}.
Let $J\colon TM \to TM$ be an almost complex structure which is compatible with $\omega$. Since the set of such structures is
contractible, we can define complex invariants of the tangent bundle.
At each fixed point $p\in M^\T$ we can define a multiset of elements $w_1,\ldots,w_n \in \ell^*\subset Lie(\T)^*$, called \textbf{weights} of the $\T$-action at $p$, which determine
the action of $\T$ on a neighborhood of $p$. Namely, there exist coordinates $z_1,\ldots,z_n$ around $p=(0,\ldots,0)$ where
the $\T$-action can be written as
\begin{equation}\label{weightsII}
\exp(\xi) \cdot (z_1,\ldots,z_n)=(e^{2\pi i w_1(\xi)}z_1,\ldots,e^{2\pi i w_n(\xi)}z_n),\quad\mbox{for all}\quad \xi\in Lie(\T)\,.
\end{equation}
Note that, since the action is required to have isolated fixed points, none of the above weights can be zero.

Moreover, we can define Chern classes.
Let 
$$c=\sum_{j=0}^n c_j\in H^{2*}(M;\Z)$$ 
be the total Chern class of the tangent bundle $(TM,J)$.  The total equivariant Chern class 
$$c^{\T}=\sum_{j=0}^n c_j^{\T}\in H^{2*}_{\T}(M;\Z)$$ of  $(TM,J)$ is defined to be
the total (ordinary) Chern class of the bundle 
$$TM\times_{\T}E\T\to M\times_{\T}E\T,$$ 
where $E\T\to B\T\simeq (\C P^\infty\times \dots \times \C P^{\infty})^{\dim(\T)}$ is the classifying bundle for $\T$. 
From naturality of equivariant Chern classes, it follows that 
$$c^{\T}(p)=\prod_{j=1}^n (1+w_j)\in H^{2*}_{\T}(\text{pt};\Z),$$ where $w_1,\ldots,w_n$ are the weights of the
$\T$-action at $p$.
In particular, 
$$c_j^{\T}(p)=\sigma_j(w_1,\ldots,w_n),$$ 
where $\sigma_j(x_1,\ldots,x_n)$ denotes the elementary symmetric polynomial of degree $j$ in $x_1,\ldots,x_n$.
Moreover, the restriction map 
$$r\colon H^*_{\T}(M;\Z)\to H^*(M;\Z),$$ 
induced by the trivial homomorphism $\{1\}\to \T$, maps $c^{\T}$ to $c$.

\subsection{Symplectic toric manifolds}\label{stm}

\begin{defin}\label{toric}
A {\bf symplectic toric manifold} is a Hamiltonian $\T$-space $(M,\omega,\psi)$, where the dimension of $\T$ is half the dimension of the manifold $M$.
\end{defin}

For any Hamiltonian $\T$-space $(M,\omega,\psi)$, the Atiyah  \cite{At1} and Guillemin-Sternberg \cite{GS82}  Convexity Theorem asserts that the image of the moment map $\psi(M)\subset Lie(\T)^*$ is a \emph{convex polytope} $\Delta$. 
If, in addition,  $(M,\omega,\psi)$ is a symplectic toric manifold, then the moment polytope $\Delta$ is Delzant (see Definition \ref{delzant def}).

By the Delzant Theorem \cite{De},  a symplectic toric manifold $(M,\omega,\psi)$ is completely determined (up to equivariant symplectomorphisms) by the  moment polytope $\Delta$. Moreover,
to each Delzant polytope $\Delta$ one can associate a symplectic toric manifold $(M_\Delta,\omega,\psi)$ such that $\psi(M_\Delta)=\Delta$. 

Choosing a splitting of the torus $\T=S^1\times \cdots \times S^1$, one can identify $Lie(\T)^*$ with $\R^n$, and $\ell^*$ with $\Z^n$, regarding $\Delta$ as a polytope in $\R^n$.
Symplectic toric manifolds satisfy the following well-known properties.
\begin{lemma}\label{properties toric 0}
 Let $(M,\omega,\psi)$ be a symplectic toric manifold and $\Delta$ its moment polytope. Then,
\begin{enumerate}
\item the moment map $\psi$ defines a bijection between the fixed point set $M^\T$ and the vertices of $\Delta$;
\item for each vertex $v$ of $\Delta$, the weights at $v$ from Definition~\ref{delzant def} are precisely the  weights of the $\T$-action at $p:=\psi^{-1}(v)$ defined in \eqref{weightsII};
\item every edge $e$ of $\Delta$ with direction vector $w\in \ell^*$ is the image of  a smoothly embedded, symplectic, $\T$-invariant $2$-sphere  $ S^2_e:=\psi^{-1}(e)\subset M$, with stabilizer $K_{e}:=\exp(\ker(w))\subset \T$, a codimension-$1$ subtorus of $\T$.
\end{enumerate}
\end{lemma}

It follows that 
$$\mathcal{S}:=\bigcup_{e\in \Delta[1]} S^2_e$$ 
is a union
of smoothly embedded, symplectic $\T$-invariant $2$-spheres. Each  of these spheres is endowed with a Hamiltonian action of the quotient circle $\T/K_{e}$ and so it is also a symplectic toric manifold. We call $\mathcal{S}$  (the
union of these spheres)  the {\bf toric 1-skeleton} \emph{of the symplectic toric manifold} $(M,\omega,\psi)$. 

The following result is well-known, and its proof can be found, for example, in \cite[Result 57]{Oda}. This is indeed a restatement, using the language of toric manifolds, of Proposition
\ref{dim2}.
\begin{prop}\label{oda}
Let $(M,\omega, \psi)$ be a $4$-dimensional symplectic toric manifold and $\Delta$ its moment polytope. Then the sum of the intersection numbers of the spheres in its toric $1$-skeleton is equal to 
$$12-3\lvert \Delta[0] \rvert.$$
\end{prop}
We are now ready to prove Theorem \ref{A} in the symplectic toric case.
\begin{thm}\label{formula c1}
Let $(M,\omega,\psi)$ be a symplectic toric manifold with toric 1-skeleton $\mathcal{S}$. Then 
\begin{equation}\label{eq: formula c1}
\sum_{e\in \Delta[1]} c_1[S_e^2]=12f_2+(5-3n)f_1\,,
\end{equation}
where $\mathbf{f}=(f_0,\ldots,f_n)$ is the $f$-vector of the moment polytope $\Delta=\psi(M)$.
\end{thm}
\begin{proof}
Denote by $\iota_{S^2_e}\colon S^2_e\hookrightarrow M$ the inclusion map. Observe that 
$$
\iota_{S^2_e}^* c_1 = c_1(T S^2_e) + c_1(\nu_{S^2_e}),
$$
where $\nu_{S^2_e}$ is the normal bundle of $S^2_e$ inside $M$, which   splits $\T$-equivariantly as a sum of $n-1$ line bundles $\nu^e_j$. Hence,
\begin{align*}
\sum_{S^2_e \in \mathcal{S}} c_1[S^2_e] & =  \sum_{S^2_e \in \mathcal{S}}  \int_{S^2_e} c_1(T S^2_e) +  c_1(\nu_{S^2_e}) = 2 f_1+  \sum_{S^2_e \in \mathcal{S}}   \int_{S^2_e} c_1(\nu_{S^2_e}) \\ & =  2 f_1 +  \sum_{S^2_e \in \mathcal{S}}  \sum_{i=1}^{n-1}  \int_{S^2_e} c_1(\nu^e_j), 
\end{align*}
where we used the fact that 
$$
\int_{S^2_e}c_1(T S^2_e) = \chi(S^2_e) = 2.
$$
For every $F\in \Delta[2]$, the preimage $\psi^{-1}(F)$ is a $4$-dimensional toric submanifold $M_{F}$ of $M$ for an appropriate subtorus of $\T$ (this is the symplectic toric counterpart of Lemma~\ref{2face}). Let $\mathcal{S}_F$ be the corresponding toric $1$-skeleton, which, of course, is a subset of $\mathcal{S}$. Moreover, since the splitting of each normal bundle $\nu_{S^2_e}$ is $\T$-invariant, each bundle $\nu^e_j$ can be identified with the normal bundle of $S^2_e$ inside a suitable $4$-dimensional submanifold $M_F$, corresponding to a $2$-face  $F\in \Delta[2]$ that has $e$ as an edge, and so will be denoted by
$$
\nu^e_j= \nu^e_F.
$$
We then have
\begin{equation}\label{an0}
\sum_{S^2_e \in \mathcal{S}}  \,\, \sum_{i=1}^{n-1}  \int_{S^2_e}  c_1(\nu^e_i) = \sum_{S^2_e \in \mathcal{S}}  \sum_{\substack{F\in \Delta[2]  \\ \text{s.t.} \, e\in F[1]}}   \int_{S^2_e}  c_1(\nu^e_F) =     \sum_{F\in \Delta[2]} \,\, \sum_{S^2_e \in \mathcal{S}_F}  \int_{S^2_e}  c_1(\nu^e_F), 
\end{equation}
where $F[1]$ is the set of edges of $F$.
On the other hand, 
$$
\sum_{S^2_e \in \mathcal{S}_F}   \int_{S^2_e}  c_1(\nu^e_F) 
$$
is the sum of the intersection numbers $S^2_e\cdot S^2_e$ for all the spheres in the toric $1$-skeleton  of the $4$-dimensional toric manifold $M_F$. So, by Proposition~\ref{oda}, it is equal to $12-3\lvert F[0]\rvert$, where $\lvert F[0] \rvert$ is the number of vertices of $F$. Consequently, 
\begin{equation}\label{an1}
 \sum_{F\in \Delta[2]} \,\, \sum_{S^2_e \in \mathcal{S}_F}   \int_{S^2_e}  c_1(\nu^e_F) =  \sum_{F\in \Delta[2]} \left( 12-3\lvert F[0]\rvert \right) = 12 f_2 - 3{n \choose 2} f_0=
 12f_2-3(n-1)f_1
\end{equation}
where we used the fact that, since $\Delta$ is simple, each vertex is in exactly $n \choose 2$ faces of dimension $2$, and
$2f_1=nf_0$.
\end{proof}
\begin{rmk}\label{analogy}
$\;$\vspace{0.1cm}
\begin{enumerate}
\item Equations \eqref{an0} and \eqref{an1} give that
\begin{equation}\label{an2}
\sum_{S^2_e \in \mathcal{S}}  \,\, \sum_{i=1}^{n-1}  \int_{S^2_e}  c_1(\nu^e_i) = 12 f_2 -3(n-1)f_1\,.
\end{equation}
This is the exact translation, into toric geometry terms, of Theorem \ref{combinatorics 2}. Indeed, one can prove that, for each $e\in E$ and each $a_i^e$ with $i=2,\ldots,n$, there exists 
a line bundle $\nu_e^j$ such that $a_i^e=\int_{S^2_e}  c_1(\nu^e_j)$, for some $j=1,\ldots,n-1$, where the $a_i^e$'s are the integers defined in Proposition \ref{int_dist}.

\item Note that, if we express \eqref{eq: formula c1} in terms of the $h$-vector of $\Delta=\psi(M)$,
we obtain exactly $C(n,\mathbf{h})$ (see \eqref{def C(n,h)}). 
\end{enumerate}
\end{rmk}

We recall an alternative characterization of the $h$-vector of a Delzant polytope, which is used in Section \ref{mg} to define
the $h$-vector in a more general context. Let $\Delta$ be an $n$-dimensional Delzant polytope, and
 $\xi\in \R^n$ a generic vector in $\R^n$, namely $\langle w,\xi \rangle \neq 0$ for every 
vector $w$ tangent to the edges of $\Delta$. Direct the edges of $\Delta$ using $\xi$, i.e.\ the edge $e$ with endpoints $v_1$ and $v_2$ is directed from $v_1$ to $v_2$
if $\langle v_2-v_1, \xi \rangle >0$. For every generic vector $\xi\in \R^n$, define the $h^{\xi}$-vector to be $\mathbf{h}^\xi=(h_0^\xi,\ldots,h_n^\xi)$, where 
\begin{equation}
h_j^\xi:=\{\# \mbox{ of vertices with }j \mbox{ entering edges}\}\quad\mbox{for all}\quad j=0,\ldots,n.
\end{equation}
\begin{lemma}\label{hequiv}
Let $\Delta$ be an $n$-dimensional Delzant polytope. Then the following three vectors associated to $\Delta$ are the same:
\begin{itemize}
\item[(1)] $\mathbf{h}=(h_0,\ldots,h_n)$
\item[(2)] $\mathbf{h}^\xi=(h_0^\xi,\ldots,h_n^\xi)$ for a generic $\xi\in \R^n$
\item[(3)] $\mathbf{b}=(b_0,b_2,\ldots,b_{2n})$, the vector of even Betti numbers of the associated (symplectic toric) manifold $M_\Delta$.
\end{itemize} 
In particular $\mathbf{h}^{\xi}$ is independent of the generic $\xi$ chosen.
\end{lemma}
\begin{proof}
Since $\Delta$ is simple, the equivalence between the $h$-vector and the $h^\xi$-vector is proved in \cite[Theorem 1.3.4]{BP}.
The proof of $\mathbf{b}=\mathbf{h}^\xi$ involves a standard argument in Morse theory, which we only sketch here (see \cite{At1}).
The function $\varphi\colon M_{\Delta}\to \R$ defined as $\varphi(p):=\langle \psi(p),\xi\rangle$ is, for a generic $\xi\in \R^n$, a Morse function with only even Morse indices and so 
it is a perfect Morse function.
Moreover, its critical points agree with the fixed points of the torus action which, in turn, are in bijection with the vertices of $\Delta$. At a critical point $p$, the Morse index
is precisely twice the number of edges entering in $v=\psi(p)$ (the orientation being that induced by $\xi$). By a standard
argument in Morse theory, we have that $b_{2j}(M_{\Delta})=h_j^\xi$ for every $j=0,\ldots,n$. 
\end{proof}

\subsection{Monotone toric manifolds}\label{mtm} 
\begin{defin}
A symplectic manifold $(M,\omega)$ is called {\bf monotone} if $c_1=r[\omega]$ for some $r\in \R$.
If, in addition, $(M,\omega,\psi)$ is a symplectic toric manifold, it is called a {\bf monotone toric manifold}.
\end{defin}  

If  $\Delta$ is a  {\bf Delzant reflexive polytope} and $(M_\Delta,\omega,\psi)$ is the corresponding symplectic toric manifold (unique up to equivariant symplectomorphisms), we have
the following  proposition  (see \cite[Prop.\ 1.8]{EP} and \cite[Sect.\ 3]{M}), which is the analogue in symplectic geometry terms of Proposition \ref{equivalent}.
We include a proof
for the sake of clarity and completeness. 
\begin{prop}\label{equivalent2}
Let $\Delta\subset \R^n$ be an $n$-dimensional Delzant polytope with $\mathbf{0}$ in its interior. Then the following conditions are equivalent:
\begin{enumerate}
\item[{\bf (i)}] $\Delta$ is a reflexive polytope;
\item[{\bf (ii)}] The symplectic toric manifold $(M_\Delta,\omega,\psi)$ with $\Delta$ as moment polytope is monotone, with $c_1=[\omega]$.
\end{enumerate}
\end{prop}
\begin{rmk}
Note that, from Proposition~\ref{equivalent},  being a Delzant reflexive polytope is equivalent to the condition that for every vertex $v$ we have 
\begin{equation}\label{eq:reflexive}
\sum_{j=1}^n w_j = -v,
\end{equation}
where $w_1,\ldots,w_n$ are the weights at $v$. This is equivalent to the fact that $\psi$ can be chosen\footnote{The moment map of a torus action is only defined up to a constant vector in $Lie(\T)^*$.}
so that 
\begin{equation}\label{chern condition 2}
c_1^\T(p)=-\psi(p)\quad\mbox{for every}\quad p\in M^\T\,,
\end{equation}
where $c_1^\T$ denotes the equivariant first Chern class of $TM$, and $c_1^\T(p)$ its restriction to the fixed point $p$.
Indeed, for every fixed point $p\in M^\T$, we have that $c_1^\T(p)$ is precisely the sum of the weights at $p$ and so the equivalence  follows from Lemma \ref{properties toric 0} (1)-(2).
\end{rmk}
\begin{proof}[Proof of Prop.\ \ref{equivalent2}]
{\bf (i)}$\implies${\bf (ii)}
By the Kirwan Injectivity Theorem \cite{Ki},  the map  
$$H^*_\T(M;\R)\to H^*_\T(M^\T;\R)$$ 
in equivariant cohomology
induced by the inclusion $M^\T \hookrightarrow M$ is always \emph{injective} for Hamiltonian torus actions. When the fixed points are isolated, one can take $\Z$ as coefficient ring. Hence, \eqref{chern condition 2} implies that 
$$c_1^\T=[\omega-\psi]\in H^2_\T(M;\Z).$$ 
(Note that here we regard $H^2_\T(M;\Z)$ as a subgroup of $H^2_\T(M;\R)$, where $[\omega-\psi]$ naturally lives.)
Since the restriction map 
$$r\colon H^*_\T(M;\R)\to H^*(M;\R)$$ 
takes $c_1^\T$ to $c_1$ and $[\omega-\psi]$ to $[\omega]$, we conclude that {\bf (i)}, which is equivalent to \eqref{chern condition 2}, implies {\bf (ii)}.

{\bf (ii)}$\implies${\bf (i)} The kernel of $r$ is precisely the ideal generated by $\mathbb{S}(Lie(\T)^*)$, the symmetric algebra on $Lie(\T)^*$. Since both $c_1$ and $[\omega]$ admit equivariant
extensions, given respectively by $c_1^\T$ and $[\omega-\psi]$, it follows  from {\bf (ii)} that 
$$c_1^\T=[\omega-\psi]+c,$$ for some constant vector $c\in Lie(\T)^*$. Hence, modulo shifting the moment map by $c$, we have that \eqref{chern condition 2}, which is equivalent to {\bf (i)}, holds. 
\end{proof}

Given an edge  $e$ of the polytope $\Delta$, it is possible to recover the symplectic volume of the $\T$-invariant $2$-sphere $S^2_e$. This is exactly
the relative length of $e$.
\begin{lemma}\label{volume}
Let $(M,\omega,\psi)$ be a symplectic toric manifold and let $\Delta$ be the corresponding moment polytope. 
Then for each edge $e\subset \Delta$ with vertices $v_1,v_2$ we have
$$v_2-v_1=Vol_\omega(S^2_e)w_e=l(e)\,w_e,$$ where $Vol_\omega(S^2_e)=\int_{S^2_e}i^*\omega$ is the symplectic volume of $S^2_e$ (with $i\colon S^2_e\hookrightarrow M$ the inclusion map) and   $w_e\in \ell^*$ is the weight of the $\T$-action on $S^2_e$ at $\psi^{-1}(v_1)$. In particular, $l(e)=Vol_\omega(S^2_e)$.
\end{lemma}
\begin{proof}
The proof of this lemma is a standard application of the Atiyah-Bott-Berline-Vergne Localization Theorem (ABBV in short)  \cite{AB,BV} to the $\T$-invariant submanifold $S^2_e$. Indeed, this sphere inherits a Hamiltonian $\T$-action from $M$, implying that
the $2$-form $i^*\omega$ can be extended to an equivariant form $i^*(\omega-\psi)$ which, by \eqref{mm}, is equivariantly closed in the Cartan complex with differential $d_{\T}=d-\sum_j\iota_{\xi_j}\otimes x^j$. Applying ABBV to $i^*(\omega-\psi)$, we obtain
$$
Vol_\omega(S^2_e)\, w_e = \left(\int_{S^2_e}i^*\omega\right) w_e = \left(\int_{S^2_e}i^*(\omega-\psi) \right)w_e=v_2-v_1\,
$$
and the  result follows from the definitions of $l(e)$ and  $w_e$. 
\end{proof}

We are now ready to give an alternative proof of Theorem~\ref{main combinatorics} that uses symplectic geometry.

\begin{proof}[Proof of Theorem~\ref{main combinatorics}]\emph{(Symplectic toric)}\label{proof symplectic}
Let $\Delta$ be a Delzant reflexive polytope and consider a symplectic toric manifold $(M_\Delta,\omega, \psi)$ such that $\psi(M_\Delta)=\Delta$. Then by Proposition~\ref{equivalent}, the manifold $M_\Delta$ is monotone with $c_1=[\omega]$. Let $\mathcal{S}$ be the toric $1$-skeleton of $M_\Delta$. Then
\begin{equation}
\sum_{e\in E} l(e) = \sum_{S^2_e \in \mathcal{S}} Vol_\omega(S^2_e) =   \sum_{S^2_e \in \mathcal{S}} [\omega]([S^2_e]) =    \sum_{S^2_e \in \mathcal{S}} c_1[S^2_e], 
\end{equation}
where we used Lemma~\ref{volume}. 
The conclusion follows from Theorem \ref{formula c1}.
\end{proof}

\section{Theorem \ref{A} and its consequences}\label{gt3}

In this Section we give the proof of Theorem \ref{A} (equivalent to Theorem \ref{formula c1} in the symplectic toric case) which uses a special behavior of the Hirzebruch genus. Theorem~\ref{A} applies to  a much broader category of spaces, namely Hamiltonian $S^1$-spaces admitting a `toric 1-skeleton'. In turns, this allows us to give a third proof of   Theorem \ref{main combinatorics}  (see page \pageref{symplectic tool}) and generalize it to some objects, called \emph{reflexive (GKM) graphs}, which behave very much like
Delzant reflexive polytopes.

\subsection{GKM spaces and toric $1$-skeletons}
Let $(M,\omega,\psi)$ be a Hamiltonian $\T$-space (see page \pageref{hamtspace}). 
When the torus acting is just a circle $S^1$, the weights at each $p\in M^{S^1}$ defined in \eqref{weightsII} are simply integers, and none of these can be zero,
since we are requiring the action to have isolated fixed points. 

Using a key property of the set of weights of a Hamiltonian $S^1$-space $(M,\omega,\psi)$, the authors in \cite{GoSa} define a family of multigraphs associated to $(M,\omega,\psi)$, called \emph{integral multigraphs} (that can actually be defined for any $S^1$-action on a compact almost complex manifold with isolated fixed points). 
Namely, if we collect all the weights for all $p\in M^{S^1}$, counting them with multiplicity, we obtain a multiset $W$ of non-zero integers with  the property  that  every time $k$ belongs to $W$ with multiplicity  $m$, then $-k$ belongs to $W$ with the same multiplicity. This fact was proved
by Hattori for almost complex manifolds \cite[Proposition 2.11]{Ha}. The possible  pairings between positive and negative weights is at the core of the definition of (integral, directed) multigraphs
associated to $(M,\omega,\psi)$, which we now recall (see  \cite[Section 4.2]{GoSa} for additional details).

\begin{defin}
A multigraph $\Gamma=(V,E)$ is a directed, integral multigraph associated to $(M,\omega,\psi)$ if it is a multigraph of degree $n=\dim(M)/2$ such that
\begin{itemize}
\item[(a)] The vertex set $V$ coincides with the fixed point set of the action, denoted by $M^{S^1}$;
\item[(b)] The edge set $E$ describes one of the possible bijections between the multisets of positive and negative weights;  namely, if there
exists an edge $e=(p,q)\in E\subset V\times V$ directed from $p$ to $q$, then one of the weights at $p$ is $k\in \Z_{>0}$, and one of the weights at $q$ is $-k$. We label the edge by $k$;
\item[(c)] For every edge $e=(p,q)$ labeled by $k>1$, both $p$ and $q$ belong to the same connected component of $M^{\Z_k}$, the submanifold of $M$ fixed by $\Z_k$
(see \cite[Lemma 4.8]{GoSa}).
\end{itemize}
\end{defin}
\begin{rmk}\label{not unique}
Note that the pairing between positive and negative weights, used in (b) to define the edge set, may not be unique and that every Hamiltonian $S^1$-space is associated to a \emph{family} of multigraphs. As an example, consider the semi-free $S^1$-action on $S^2\times S^2$ described in \cite[Example 4.13]{GoSa}, with $a=b=1$. This action has four possible associated integral multigraphs (see \cite[Figure 4.1]{GoSa}).  
\end{rmk}

The definition of a multigraph associated to $(M,\omega,\psi)$ is inspired by the edge set of the Delzant polytope associated to
a symplectic toric manifold or, more generally, by the
 \emph{GKM graph} associated to a GKM space. 
 \begin{defin}\label{GKM def}
Let $(M,\omega,\psi)$ be a Hamiltonian $\T$-space, with $\dim(\T)>1$. The $\T$-action is called \emph{{\bf GKM} (Goresky-Kottwitz-MacPherson  \cite{GKM})} - and the triple $(M,\omega,\psi)$ is called a {\bf(Hamiltonian) GKM space} - if, for every codimension-$1$ subgroup $K\subset \T$, the fixed submanifold
$M^K$ has dimension at most $2$. 
\end{defin}

Note that the closure of each $2$-dimensional component fixed by a codimension-$1$ subtorus $K$ is a symplectic surface smoothly embedded in $M$ (see the proof of Lemma \ref{s1 one skeleton} for details), inheriting an effective Hamiltonian action of the circle $\T/K$. Hence  it must be a sphere on which  $\T$ acts with  exactly two fixed points and,
in analogy with the toric case, we have the following definition.
\begin{defin}\label{gkm spheres}
Given a GKM space $(M,\omega,\psi)$, the union of the smoothly embedded, symplectic, $\T$-invariant spheres fixed by codimension-$1$ tori is called the {\bf toric 1-skeleton} of $(M,\omega,\psi)$.
\end{defin}

The ``GKM condition" can be rephrased in terms of the weights of the $\T$-action at the fixed points. 
Indeed, from \eqref{weightsII} it is easy to see that the action is GKM if and only if for each fixed point $p$ the weights at $p$ are pairwise linearly independent.  
Hence, (D3) and Lemma \ref{properties toric 0} (2) imply that symplectic toric manifolds are  GKM spaces. However, not every Hamiltonian GKM space is toric because in general
$$\frac{\dim(M)}{2}-\dim(\T)\geq 0.$$
Moreover, the spheres in the toric 1-skeleton of the GKM space $(M,\omega,\psi)$ may not be in one-to-one correspondence with the edges of the convex polytope $\psi(M)$, as the image of some of these spheres may be in the interior of $\psi(M)$ (see Figure \ref{flags pic}).

\begin{figure}[htbp]
\begin{center}
\includegraphics[width=9cm]{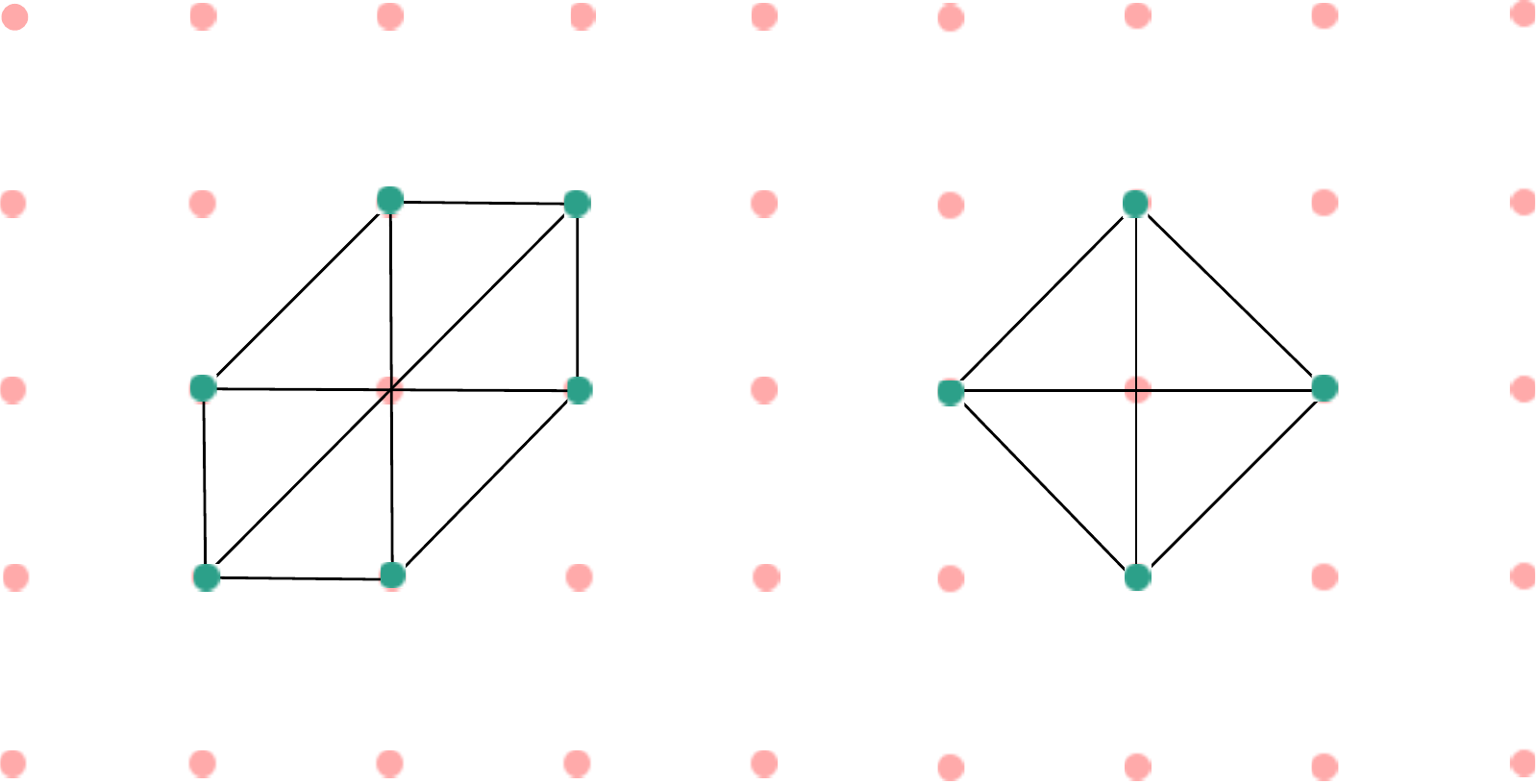}
\caption{Two examples of GKM graphs embedded in $Lie(\T^2)^*\simeq \R^2$ with lattice $\Z^2$ corresponding to
coadjoint orbits of type A and B. 
The vertices are marked in blue,
and the image of the moment map is the convex hull of the vertices.}
\label{flags pic}
\end{center}
\end{figure}

Therefore, the moment map image $\psi(M)$ is no longer enough to  keep track of information on the intersection properties of these spheres and their stabilizers. So, to every GKM space $(M,\omega,\psi)$ one associates a (directed, labeled) graph $\Gamma_{GKM}=(V,E_{GKM})$, called \emph{(directed, labeled) GKM graph}. This is defined as follows. 
First, pick a generic vector $\xi\in Lie(\T)$, i.e.\  $\langle w,\xi\rangle\neq 0$ for every $w\in \ell^*$ that occurs as a weight of the $\T$-action at a fixed point. 
Assuming that $\xi$ generates a circle subgroup $C$ in $\T$, let $\varphi:=\psi^{\xi}\colon M\to \R$ be the $\xi$-component of the moment map, namely 
$$\varphi(\cdot)=\langle \psi(\cdot), \xi \rangle.$$ 
Note that
this is a moment map for the action of $C$ on $(M,\omega)$ with isolated fixed points. 

\begin{defin}\label{def:GKMgraph}
The {\bf GKM graph}  associated to a GKM space  $(M,\omega,\psi)$ is the graph  $\Gamma_{GKM}=(V,E_{GKM})$ such that:
\begin{itemize}
\item[(a')] The vertex set $V$ coincides with $M^\T$; 
\item[(b')]  The edge set $E_{GKM}\subset V\times V$ describes the intersection properties of the $2$-spheres fixed by a codimension-$1$ subtorus. In particular, there
exists an edge $e=(p,q)$ from $p$ to $q$ precisely if there exists a $2$-sphere fixed by a codimension-$1$ subtorus of $\T$,
where $\T$ acts with  fixed points $p$ and $q$, and $\varphi(p)<\varphi(q)$.
\item[(c')] Each edge $e=(p,q)\in E_{GKM}$ with associated sphere $S^2_e$ is labeled by the weight $w_e\in \ell^*$ of the $\T$-action on $S^2_e$ at $p$. 
\end{itemize}
\end{defin}

\begin{rmk}\label{S1-GKM}
$\;$\vspace{0.1cm}
\begin{enumerate}
\item[(i)] The GKM graph of a GKM-space $(M,\omega,\psi)$ is automatically $n$-valent, where $n=\dim(M)/2$.
\item[(ii)] In contrast with  Hamiltonian $S^1$-spaces, every GKM space $(M,\omega,\psi)$ has a \emph{unique} GKM graph (see Remark \ref{not unique}). 
\item[(iii)] Suppose that, given a Hamiltonian $S^1$-space $(M,\omega,\psi)$, the $S^1$-action extends to a Hamiltonian GKM action of a torus $\T$  on $(M,\omega)$,  and
let $\xi$ be a primitive generic vector in $\ell$ (the weight lattice of $\T$) so that
$\exp(\R \xi)=S^1$. Then one of the possible multigraphs associated to $(M,\omega,\psi)$ is the GKM graph of the $\T$-action directed by $\xi$. This is labeled as follows: If an edge $e$ of the GKM graph is labeled by $w_e\in \ell^*$, then $e$, regarded as an edge of the multigraph associated to the Hamiltonian $S^1$-space $(M,\omega,\psi)$,  is labeled by $w_e(\xi)$.
\end{enumerate} 
\end{rmk}

\begin{rmk}\label{embed gkm}
Using the moment map, one can map the GKM graph into $Lie(\T)^*$ in the following way:
\begin{itemize}
\item[(a'')] A vertex $v$, corresponding to $p\in M^\T$ is mapped to the corresponding value of the moment map $\psi(p)$; 
\item[(b'')] An edge $e=(p,q)$, corresponding to a sphere $S^2_e$ is mapped to the line segment from $\psi(p)$ to $\psi(q)$, which is exactly $\psi(S^2_e)$. 
\item[(c'')] If $\ell^*$ is the dual lattice in $Lie(\T)^*$, the weight $w_e$ associated to the edge $e=(p,q)$ is the vector with the same direction of $\psi(q)-\psi(p)$ which is primitive in 
$\ell^*$.
\end{itemize}
\end{rmk}

The concept of relative length in Definition \ref{int length} can be easily generalized to line segments in $Lie(\T)^*$ endowed with lattice $\ell^*$.

If $e\in E_{GKM}$ is an edge of a GKM graph, it is again possible to recover the symplectic volume of the corresponding sphere in the toric 1-skeleton, in analogy
with Lemma \ref{volume}. 
\begin{lemma}\label{volume2}
Let $(M,\omega,\psi)$ be a GKM space, and let $(V,E_{GKM})$ be its labeled GKM graph. 
Then, for each edge $e=(p,q)\in E_{GKM}$, we have
$$\psi(q)-\psi(p)=Vol_\omega(S^2_e)w_e=l(e)\,w_e,$$ where $Vol_\omega(S^2_e)=\int_{S^2_e}i^*\omega$ is the symplectic volume of $S^2_e$ and $i\colon S^2_e\hookrightarrow M$ is the inclusion map. 
\end{lemma}
The proof is, \emph{mutatis mutandis}, the same as that of Lemma \ref{volume}.
\begin{rmk}\label{partial embed}
Note that Lemma \ref{volume2} and the ``GKM condition" (the weights of the action at each fixed point are pairwise linearly independent) imply that
there is exactly one edge $e$ of the GKM graph with endpoints $p$ and $q$. Hence, if $\psi(S^2_e)=\psi(S^2_f)$, then $S^2_e=S^2_f$, for every $e,f\in E_{GKM}$.
\end{rmk}

Symplectic toric and GKM spaces motivate the following definition, now for Hamiltonian $S^1$-spaces.
\begin{defin}\label{smooth 1-skeleton}
Let $(M,\omega,\psi)$ be a Hamiltonian $S^1$-space. 
\begin{enumerate}
\item We say that $(M,\omega,\psi)$ {\bf admits a toric 1-skeleton} if
there exists an integral multigraph $\Gamma=(V,E)$ associated to $(M,\omega,\psi)$ satisfying the following property:
For each edge $e=(p,q)\in E$ labeled by $k\in \Z_{>0}$ there exists a smoothly embedded, symplectic, $S^1$-invariant sphere fixed by $\Z_k$, where $S^1$ acts with fixed points $p$ and $q$.  
\item 
If such a multigraph $\Gamma$ exists, the set of spheres $\mathcal{S}=\{S^2_e\}_{e\in E}$ obtained by picking for each edge $e\in E$ exactly one sphere $S^2_e$ satisfying the properties in \emph{(1)} is called the {\bf toric 1-skeleton} of $(M,\omega,\psi)$ associated to $\Gamma$. 
\end{enumerate}
\end{defin} 
The idea behind the toric 1-skeleton is the following. Consider an $S^1$-invariant metric on $M$, and let $\operatorname{grad}\psi$ be the gradient of $\psi$ w.r.t. this metric.
Then the $\R$-action associated to the flow of $\operatorname{grad}\psi$ commutes with the $S^1$ action, giving rise to a $\C^*=S^1\times \R^+$-action on $M$. 
For each $p\in M$, the closure of the $\C^*$-orbit through $p$ is an embedded, symplectic ($S^1$-invariant) $2$-sphere, not necessarily smooth at the poles. The toric 1-skeleton exists if one can pick a subset of such spheres
satisfying the properties in (1) (see \cite[Section 3]{AH}).  

The name follows from the  fact that, if such spheres exist, they are
symplectic sub\-mani\-folds endowed with a Hamiltonian circle action with stabilizer $\Z_k$, and so  they admit an effective Hamiltonian action of the circle $S^1/\Z_k$, turning each of them into a symplectic toric (sub)manifold.
Also note that from (2) the cardinality of $\mathcal{S}$ is finite and equal to 
$$|E|=\displaystyle\frac{\dim(M)}{4} \,\vert M^{S^1}\rvert=\displaystyle\frac{n}{2}\,\chi(M),$$ where $\chi(M)$ denotes the Euler characteristic\footnote{Indeed $|M^{S^1}|=\chi(M)$ holds for every compact almost complex manifold with an $S^1$ with isolated fixed points.} of $M$ and $n=\dim (M)/2$. 
Hence, Hamiltonian $S^1$-spaces admitting a toric 1-skeleton have a special finite subset of $S^1$-invariant symplectic spheres. 
The following Lemma gives us some examples  of Hamiltonian $S^1$-spaces admitting a toric 1-skeleton. 

\begin{lemma}\label{s1 one skeleton}
Let $(M,\omega,\varphi)$ be a Hamiltonian $S^1$-space. If $(M,\omega,\varphi)$  satisfies one of the following conditions, then it admits a toric 1-skeleton:
\begin{enumerate}
\item[(i)] the $S^1$-action extends to a GKM action, (or, in particular, to a symplectic toric action);
\item[(ii)] none of the weights is equal to $1$ and,  at each fixed point $p\in M^{S^1}$, the weights are pairwise prime;
\item[(iii)] the (real) dimension of $M$ is $4$.
\end{enumerate}
\end{lemma}
\begin{proof}
(i) Let us assume that  the action of $S^1$ extends to a GKM-action of a torus $\T$ and consider its GKM toric $1$-skeleton $\mathcal{S}$ as in Definition~\ref{gkm spheres} and its GKM graph $\Gamma_{GKM}=(V,E_{GKM})$. Each sphere in $\mathcal{S}$ is $\T$-invariant, and therefore is also $S^1$-invariant. Moreover, since  $S^1$ has a discrete fixed point, set it does not fix any sphere in $\mathcal{S}$, and every connected component of $M^{S^1}$ (isolated points) is $\T$-invariant. Consequently, $S^1$ has exactly two fixed points on each sphere, which must coincide with the two $\T$-fixed points and $M^{S^1}=M^\T$. 

Let us take $\xi\in Lie(\T)$ such that $S^1=\{\textrm{exp} (t\xi)\mid t \in \R \}$. Since $S^1$ has a discrete fixed point set, $\langle w,\xi\rangle\neq 0$ for every $w\in \ell^*$ that occurs as a weight of the $\T$-action at a fixed point. If $S$ is a sphere in $\mathcal{S}$ with associated edge $e=(p,q)\in E_{GKM}$ labeled by the weight $w_e\in \ell^*$, then $\langle w_e,\xi\rangle$ is the $S^1$-isotropy weight at $p$, while $-\langle w_e,\xi\rangle$ is the one at $q$. Moreover, it is clear that if $\lvert \langle w_e,\xi\rangle \rvert =k >1$, then both $p$ and $q$ belong to the same connected component of $M^{\Z_k}$  (the set of points fixed by the finite subgroup $\Z_k\subset S^1$). 

We conclude  that $\Gamma_{GKM}$ is an integral multigraph $\Gamma=(E,V)$ for $(M,\omega,\varphi)$ and, for each edge $e=(p,q)\in E$ labeled by $k=\lvert \langle w_e,\xi\rangle \rvert \in \Z_{>0}$, there exists a smoothly embedded sphere fixed by $\Z_k$ (the corresponding sphere in $\mathcal{S}$), where $S^1$ acts with fixed points $p$ and $q$.

(ii) If none of the weights is equal to $1$ and at each fixed point the weights are pairwise prime then, for each isotropy group $\Z_k\subset S^1$, the connected components of the manifolds $M^{\Z_k}$ are  smoothly embedded closed symplectic 2-spheres (here called $\Z_k$-spheres). Each of these spheres contains exactly two fixed points. Moreover, for $k\geq 2$, a fixed point has weight $-k$ if and only if it is the north pole of a $\Z_k$-sphere and it has weight $k$ if and only if it is the south pole of a $\Z_k$-sphere.  Hence, there is a pairing between the positive and negative weights resulting in a multigraph for the $S^1$-action that satisfies Definition~\ref{smooth 1-skeleton}, and the result follows.

(iii) If $\dim (M)=4$ then, since $S^1$ has a discrete fixed point set, it extends to a $\T^2$ toric action \cite[Theorem 5.1]{K}, and the result follows from (i).

\end{proof}
\begin{rmk} 
$\;$\vspace{0.1cm}
\begin{itemize}
\item[(a)] Note that there are no examples known of Hamiltonian $S^1$-spaces with a discrete fixed point set  that do not admit a toric 1-skeleton. 
\item[(b)] Observe that in Lemma~\ref{s1 one skeleton} (i) and (ii) the smoothness of the spheres is ensured by the fact that each of them is a connected component of an isotropy submanifold, i.e. the set of
points in $M$ fixed by some subgroup of the torus. 
\item[(c)] In \cite{K}, Karshon exhibits an algorithm to extend a Hamiltonian $S^1$-action with isolated fixed points on a $4$-dimensional symplectic manifold to a toric $\T^2$-action.
Such extension is not unique. From the discussion about the toric 1-skeleton of symplectic toric manifolds we deduce that the same Hamiltonian $S^1$-space may admit more than
one toric 1-skeleton.
\end{itemize}
\end{rmk}

An important  topological property of the toric 1-skeleton comes from Poincar\'e duality.
\begin{lemma}\label{P Dual}
Let $(M,\omega,\psi)$ be a Hamiltonian $S^1$-space admitting a toric 1-skeleton $\mathcal{S}=\{S^2_e\}_{e\in E}$. Then the Chern class $c_{n-1}\in H^{2(n-1)}(M;\Z)$ is the Poincar\'e dual to the class of $\mathcal{S}\in H_2(M;\Z)$.
\end{lemma}
\begin{proof} Let us consider a class $\alpha\in H^2(M;\Z)$. 
We need to show that
\begin{equation}\label{wtp}
\sum_{e\in E}\int_{S^2_e} i^* \alpha = \int_M \alpha \smile c_{n-1},
\end{equation}
where  $i:\mathcal{S} \to M$ is the inclusion map.
By the Kirwan Surjectivity Theorem \cite{Ki}, we know that the restriction map 
$$r\colon H^2_{S^1}(M;\Z)\to H^2(M;\Z)$$ 
is surjective (note that the fixed points are isolated, hence we can take $\Z$-coefficients). Hence one can consider an equivariant extension 
$\alpha^{S^1}\in H^2_{S^1}(M;\Z)$  of $\alpha$ and, by dimensional reasons, we have
$$\int_{S^2_e} i^* \alpha^{S^1}=\int_{S^2_e} i^*\alpha.$$ We  can  compute the integral on the LHS of \eqref{wtp} using the ABBV Localization Theorem. Let $p,q$ be the fixed points of the $S^1$ action
on $S^2_e$, and $w_e$ the weight of the isotropy action at $p$. Then ABBV
 gives 
 $$\int_{S^2_e}i^*\alpha=\frac{\alpha^{S^1}(p)-\alpha^{S^1}(q)}{w_e}$$  
and so
\begin{equation}\label{abbv}
\sum_{e\in E}\int_{S^2_e}i^*\alpha = \sum_{e=(p,q)\in E} \frac{\alpha^{S^1}(p)-\alpha^{S^1}(q)}{w_e}\,.
\end{equation}
We can apply the same procedure to compute the integral on the RHS of \eqref{wtp}. 
Note that, if $S_e^2$ is a sphere in $\mathcal{S}$,  the weights of the $S^1$-representations on the tangent spaces at the two $S^1$-fixed points $p,q\in S_e^2$ are given respectively by $w_{1,p},w_{2,p},\ldots, w_{n,p}$ and $w_{1,q},w_{2,q},\ldots, w_{n,q}$, where $w_{1,p}=w_e$ and     $w_{1,q}=-w_e$.
Considering equivariant extensions $\alpha^{S^1}$ and $c_{n-1}^{S^1}$ of $\alpha$ and $c_{n-1}$ we again have,  for dimensional reasons, 
$$\int_M \alpha \smile c_{n-1}=\int_M \alpha^{S^1}\smile c_{n-1}^{S^1}.$$ Hence
ABBV gives
\begin{align*}
 \int_\M \alpha \smile c_{n-1} & = \sum_{p\in M^{S^1}} \frac{\alpha^{S^1}(p) c^{S^1}_{n-1}(p)}{\prod_{j=1}^n w_{j,p}} = \sum_{p\in M^{S^1}} \frac{\alpha^{S^1}(p) \left(\sum_{l=1}^n \prod_{k\neq l} w_{k,p}\right) }{\prod_{j=1}^n w_{j,p}} \\ 
 & = \sum_{p\in M^{S^1}}\sum_{k=1}^{n} \frac{\alpha^{S^1}(p)}{w_{k,p}} =  \sum_{e=(p,q)\in E} \frac{\alpha^{S^1}(p)-\alpha^{S^1}(q)}{w_e}\\
\end{align*}
and \eqref{wtp} follows.
\end{proof}
\begin{rmk}\label{remark lemma}
$\;$\vspace{0.1cm}
\begin{itemize}
\item[(a)] Lemma \ref{P Dual} is already known for (symplectic) toric manifolds. However it still holds true for the much broader category of Hamiltonian $S^1$-spaces admitting a toric 1-skeleton.
\item[(b)] For a symplectic toric manifold $(M,\omega,\psi)$ it is well known that the $j$-th Chern class  is the Poincar\'e Dual class to the sum of the fundamental classes of the preimages of $(n-j)$-dimensional faces of 
the moment polytope $\Delta$ (see \cite[Corollary 11.5]{Da}).
However, the existence of just a circle action does not allow us, in general, to define toric skeleta of higher dimensions.
\end{itemize}
\end{rmk}

\subsection{Proof of Theorem \ref{A}}\label{subsec:thmA}
Let $\mathcal{S}=\{S^2_e\}_{e\in E}$ be a toric 1-skeleton associated to $(M,\omega,\psi)$.
From Lemma \ref{P Dual}, applied to $\alpha=c_1$, we obtain
\begin{equation}\label{c1cn-1}
\sum_{e\in E}c_1[S^2_e]=\int_M c_1c_{n-1}\,,
\end{equation} 
where $c_1[S^2_e]$ denotes $\int_{S^2_e}c_1$.
So \cite[Corollary 3.1]{GoSa} gives
\begin{equation}\label{sum volumes}
\sum_{S^2_e\in \mathcal{S}} c_1\big[S^2_e\big]=\sum_{j=0}^n b_{2j}(M)\Big[6j(j-1)+\frac{5n-3n^2}{2}\Big]\,.
\end{equation}
Observe that the moment map $\psi$ is a perfect Morse function, hence the odd Betti numbers of $M$ vanish and $\chi(M)=\sum_{j=0}^n b_{2j}(M)$. Moreover,
Poincar\'e duality implies $b_{2j}(M)=b_{2(n-j)}(M)$ for every $j=0,\ldots,n$.
Set $g(j,n)=6j(j-1)+\frac{5n-3n^2}{2}$.
For $n$ even, the right hand side of \eqref{sum volumes} becomes
\begin{equation}
\begin{aligned}
\sum_{j=0}^n b_{2j}(M)g(j,n)= & -\frac{n}{2}b_n(M)+ \sum_{k=1}^{\frac{n}{2}} b_{n-2k}(M)\left[g\left(\frac{n}{2}-k,n\right)+g\left(\frac{n}{2}+k,n\right)\right]\\
 = & \;  -\frac{n}{2}b_n(M)+2\sum_{k=1}^\frac{n}{2} \left( 6k^2-\frac{n}{2}\right)b_{n-2k}(M) \\
 = & \; 12\displaystyle\sum_{k=1}^{\frac{n}{2}}( k^2b_{n-2k}(M)) -\frac{n}{2}\left[b_n(M)+2\sum_{k=1}^{\frac{n}{2}}b_{n-2k}(M)\right] \\
 = & \; 12\displaystyle\sum_{k=1}^{\frac{n}{2}} \Big[ k^2b_{n-2k}(M)\Big] - \frac{n}{2}\chi(M),
\end{aligned}
\end{equation}
where we used that $\chi(M)=\displaystyle b_n(M)+2\sum_{k=1}^{\frac{n}{2}}b_{n-2k}(M)$,
and the claim follows.

For $n$ odd, the right hand side of \eqref{sum volumes} becomes
\begin{equation}
\begin{aligned}
\sum_{j=0}^n b_{2j}(M)g(j,n)= & \sum_{k=0}^{\frac{n-1}{2}} b_{n-1-2k}(M)\left[g\left(\frac{n-1}{2}-k,n\right)+g\left(\frac{n-1}{2}+k+1,n\right)\right]\\
 = & \; 2\sum_{k=0}^\frac{n-1}{2} \left[ 6k(k+1)-\frac{n-3}{2}\right]b_{n-1-2k}(M) \\
 = & \; 12\displaystyle\sum_{k=1}^{\frac{n-1}{2}}\Big[ k(k+1)b_{n-1-2k}(M)\Big] -\left(\frac{n-3}{2}\right)\chi(M)\,,
\end{aligned}
\end{equation}
where we used that $\chi(M)=2\displaystyle\sum_{k=0}^{\frac{n-1}{2}}b_{n-1-2k}(M)$,
and the claim follows.

\begin{rmk}\label{rmk A}
$\;$\vspace{0.1cm}
\begin{itemize}
\item[(i)] There are three key facts needed to prove Theorem \ref{A} in its generality. First, \eqref{c1cn-1} holds for Hamiltonian $S^1$-spaces admitting a toric 1-skeleton. 
Then the integral of $c_1c_{n-1}$ only depends on the Hirzebruch genus \cite[Theorem 2]{Sal}, and, finally, for Hamiltonian $S^1$-spaces the Hilzebruch genus is
\emph{rigid}, depending only on the Betti numbers of $M$.  

\item[(ii)]  Although a Hamiltonian $S^1$-space may admit more than one toric 1-skeleton, the sum of the integrals of $c_1$
on the corresponding spheres is independent of the toric 1-skeleton chosen, and so it is an invariant of $(M,\omega,\psi)$. 
\end{itemize}
\end{rmk}

\begin{corollary}\label{cor 1}
Under the same hypotheses of Theorem \ref{A}, as special cases we have:
\begin{itemize}
\item[$\bullet$] If $n=2$ then 
\begin{equation}\label{n=2 sym}
\sum_{S^2_e\in \mathcal{S}}c_1[S^2_e]+\chi(M)= 12\,.
\end{equation}
\item[$\bullet$] If $n=3$ then
\begin{equation}\label{n=3 sym}
\sum_{S^2_e\in \mathcal{S}}c_1[S^2_e]= 24\,.
\end{equation}
\end{itemize}
\end{corollary}
We observe the following important facts:
\begin{rmk}
$\;$\vspace{0.1cm}
\begin{enumerate}
\item[(i)] Equation \eqref{n=2 sym} can be regarded as a generalization of \emph{Noether's formula} to Hamiltonian $S^1$-spaces of (real) dimension $4$,
which, by Lemma \ref{s1 one skeleton} (iii), always admit a toric 1-skeleton.
\item[(ii)] Corollary \ref{cor 1} may be regarded as a symplectic analogue of the ``12" and ``24" theorem for reflexive polytopes (see Corollary \ref{cor combinatorics}).
\end{enumerate}
\end{rmk}

We can now provide the third more general proof of Theorem~\ref{main combinatorics}.

\begin{proof}[Proof of Theorem~\ref{main combinatorics}]\emph{(Symplectic)}\label{proof symplectic tools} The proof is identical to the symplectic toric proof of Theorem~\ref{main combinatorics} in page~\pageref{proof symplectic}. One just has to replace Theorem~\ref{formula c1} with Theorem~\ref{A}.
\end{proof}

\begin{rmk}\label{symplectic tool}
The expression  involving  the Hirzebruch genus in \eqref{eq: hirzebruch} is the key tool to prove \eqref{sum volumes}.
\end{rmk}

\section{Consequences in symplectic geometry and combinatorics}
\subsection{Monotone Hamiltonian spaces: indices and Betti numbers}\label{mhs}
\begin{defin}\label{monotone} A {\bf monotone Hamiltonian $S^1$-space} is a
 Hamiltonian $S^1$-space $(M,\omega,\psi)$ with $c_1=r[\omega]$, for some $r\in\R$.
\end{defin}  
\begin{lemma}\label{lambda positive}
Let $(M,\omega,\psi)$ be a monotone Hamiltonian $S^1$-space. Then $c_1=r[\omega]$, with $r>0$.
\end{lemma}
\begin{proof}
The key ingredient to prove that $r>0$ is the fact that the action is Hamiltonian.   
The first Chern class $c_1$ always admits an equivariant extension $c_1^{S^1}\in H^2_{S^1}(M;\Z)$ and, in this case, since the action is Hamiltonian, so does $[\omega]$: an equivariant extension
is given precisely by  $[\omega-\psi\otimes x]$.
Moreover, by a theorem of Kirwan, the restriction map $H^2_{S^1}(M;\R)\to H^2(M;\R)$ is surjective, and the kernel is the ideal generated by $x$, where $x$ is the generator of the integral equivariant cohomology ring of a point, namely 
$$H^*(BS^1;\Z)=H^*(\C P^\infty;\Z)=\Z[x].$$ 
It follows that $c_1^{S^1}=r[\omega-\psi\otimes x]+a\,x$, for some $a\in \R$.
To conclude that $r>0$ it is sufficient to evaluate the previous expression at the minimum $p_{min}$ and maximum $p_{max}$
of the moment map $\psi$, which are also fixed points of the action. Our conventions imply that all the weights at the minimum (resp.\ maximum) are positive (resp.\ negative). Moreover, $c_1^{S^1}(p)=(\sum_{j}w_j)x$, where $w_1,\ldots,w_n$ are the weights of the $S^1$-action at $p$. Hence we have
$$
-r\, \psi(p_{min})=\frac{c_1^{S^1}(p_{min})}{x}> \frac{c_1^{S^1}(p_{max})}{x}=-r\, \psi(p_{max})
$$
and the conclusion follows.
\end{proof}
From Lemma~\ref{lambda positive}, for monotone Hamiltonian $S^1$-spaces, the symplectic form can be rescaled so that $c_1=[\omega]$.

Theorem \ref{A} implies the existence of inequalities relating the Betti numbers and the \emph{index} of a monotone Hamiltonian $S^1$-space admitting a toric 1-skeleton.
The concept of index is inspired by the analogue in algebraic geometry, and is defined as follows.
\begin{defin}\label{psindex}
Let $(M,\omega)$ be a compact symplectic manifold, and let $c_1$ be the first Chern class of $(TM,J)$.
 The {\bf index} $k_0$ is the largest integer such that $c_1=k_0 \eta$ for some non-zero element $\eta\in H^2(M;\Z)$, modulo torsion elements. 
\end{defin}  

Since a Hamiltonian $S^1$-space (which has isolated fixed points) is simply connected (see \cite{Li2}), the index coincides with the \emph{minimal Chern number}, 
i.e.\ the integer $N$ such that $\langle c_1,\pi_2(M)\rangle = N \Z$ (see \cite[Remark 3.13]{Sa}).
Moreover, it satisfies 
\begin{equation}\label{bound k0}
1\leq k_0 \leq n+1
\end{equation}
(see \cite[Corollary 1.3]{Sa}). 
For monotone  toric manifolds the index can be easily recovered from the image of the moment map.
\begin{prop}\label{index gcd}
Let $(M,\omega,\psi)$ be a monotone toric manifold with moment polytope $\Delta$ and $c_1=[\omega]$. Then
\begin{equation}\label{gcd}
k_0=\operatorname{gcd}\{l(e)\mid e\in \Delta[1]\}\,.
\end{equation}
\end{prop}
\begin{proof}
From the definition of $k_0$ it is clear that $l(e)$ is a multiple of $k_0$ for all $e\in \Delta[1]$. Moreover,
since the spheres in the toric 1-skeleton generate $H_2(M;\Z)$, we have $\operatorname{gcd}\displaystyle\{l(e)/k_0\}_{e\in \Delta[1]}=1$, and the conclusion follows.
\end{proof}
 
 It is natural to ask whether there exists a relation between the index and the Betti numbers of $(M,\omega)$. This question is inspired by the long-standing \emph{Mukai conjecture} for Fano varieties \cite{Muk} and its generalizations, and has been extensively studied in the algebraic geometric setting (see for instance \cite{BCDD, Andreatta, C, CJR}).
In the next corollary we prove that when $(M,\omega,\psi)$ is a monotone Hamiltonian $S^1$-space admitting a toric 1-skeleton, there are inequalities relating the index and the Betti numbers. As one may expect, such inequalities imply stronger restrictions
when the index is high.

\begin{corollary}\label{cor 2}
Let $(M,\omega)$ be a compact, connected symplectic manifold of dimension $2n$ with index $k_0$, and let $\mathbf{b}=(b_0,\ldots,b_{2n})$ be the
vector of its even Betti numbers. 
Consider the integer $C(k_0,n,\mathbf{b})$ defined as 
$$
C(k_0,n,\mathbf{b}):=
\begin{cases}
\displaystyle\sum_{k=1}^{\frac{n}{2}}\Big[12k^2-n(k_0+1)\Big]b_{n-2k}- \frac{n}{2}(k_0+1)\,b_n, \mbox{ for }n\mbox{ even}\\
& \\
\displaystyle\sum_{k=1}^{\frac{n-1}{2}}\Big[12k(k+1)+3-n(k_0+1)\Big]b_{n-1-2k}- \big[n(k_0+1)-3\big]b_{n-1}, \mbox{ for }n\mbox{ odd.}
\end{cases}
$$
If $(M,\omega,\psi)$ is a monotone Hamiltonian $S^1$-space which admits a toric 1-skeleton $\mathcal{S}$, then 
$C(k_0,n,\mathbf{b})$ is a \emph{non-negative} multiple of $k_0$. 

Moreover $C(k_0,n,\mathbf{b})$ vanishes if and only if $c_1[S_e^2]=k_0$ for all $S_e^2\in \mathcal{S}$.
\end{corollary}
We remark that, thanks to Lemma \ref{s1 one skeleton}, the corollary above applies to all monotone toric manifolds (or more generally to all monotone GKM spaces).
\begin{proof}
Let $\mathcal{S}=\{S^2_e\}_{e\in E}$ be the toric 1-skeleton associated to a suitable multigraph $\Gamma=(V,E)$. 
Observe that, for each $e\in E$, the integral of $c_1$ on $S^2_e$ is an integer. Moreover, since the symplectic form can be taken so that $c_1=[\omega]$, and since each of these spheres is symplectic, this integral must be a positive integer. Finally, since the index is $k_0$, this integer $c_1[S_e^2]$ must be a positive multiple of $k_0$.
Thus 
\begin{equation}\label{eq mon}
\sum_{e\in E}c_1[S^2_e] - k_0|E|=\sum_{e\in E}c_1[S^2_e]-k_0\frac{n}{2}\chi(M)\,
\end{equation}
is a non-negative multiple of $k_0$, and it is zero precisely if $c_1[S_e^2]=k_0$ for all $S_e^2\in \mathcal{S}$. 

For $n$ even, Theorem \ref{A} implies that the right-hand side of \eqref{eq mon} is
$$
12\displaystyle\sum_{k=1}^{\frac{n}{2}}\Big[ k^2b_{n-2k}(M)\Big] -\frac{n}{2}(k_0+1)\chi(M)
$$
and the claim follows easily from the definition of $C(k_0,n,\mathbf{b})$, and the fact that, for $n$ even, 
$$\chi(M)=b_n(M)+2\displaystyle\sum_{k=1}^{\frac{n}{2}}b_{n-2k}(M).$$
The case $n$ odd is similar, and the details are left to the reader.
\end{proof}

\renewcommand{\arraystretch}{1.5}
\begin{table}[h]
\begin{center}
  \begin{tabular}{ | c || l | l | l | l | }
    \hline
    $k_0$ & $n=2$ & $n=3$ & $n=4$ & $n=5$ \\ \hline \hline
    $1$ & \cellcolor{Apricot}$2(4-b_2)$ & $\cellcolor{Apricot}3(7-b_2)$ & $4(10+b_2-b_4)$ & $65+17b_2-7b_4$ \\ \hline
    $2$ & \cellcolor{SpringGreen}$3(2-b_2)$ & $\cellcolor{Apricot}6(3-b_2)$ & $6(6-b_4)$ & $2(25+6b_2-6b_4)$\\ \hline
    $3$ &  \cellcolor{SpringGreen}$4(1 - b_2)$ &  \cellcolor{SpringGreen}$3(5-3b_2)$ & $\cellcolor{Apricot}4(8-b_2-2b_4)$ & $55+7b_2-17b_4$\\ \hline
   $4$ &  & $\cellcolor{SpringGreen}12(1-b_2)$ & $\cellcolor{SpringGreen}2(14-4b_2-5b_4)$ & $2(25+b_2-11b_4)$\\ \hline
  $5$ & & & $\cellcolor{SpringGreen}12(2-b_2-b_4)$ & $\cellcolor{Apricot}3(15-b_2-9b_4)$\\ \hline
   $6$ & & & & $\cellcolor{SpringGreen}8(5-b_2-4b_4)$\\ 
    \hline
  \end{tabular}
 \end{center}
  \caption{
  Values of $C(k_0,n,\mathbf{b})$ for $2\leq n\leq 5$.
 The colored cells correspond to values of $n$ and $k_0$ where the only positive coefficient
in $C(k_0,n,\mathbf{b})$ is that of $b_0=1$.}
\label{table}

\end{table}

In Table \ref{table} we write the values of $C(k_0,n,\mathbf{b})$ for all $2\leq n \leq 5$ and $1\leq k_0 \leq n+1$
in terms of $\mathbf{b}$, where we used that $b_0=1$.
In the green cells we list the cases where, as a consequence of Corollary \ref{cor 2}, the Betti numbers are
completely determined by $n$ and $k_0$. In the orange cells we list the cases where there are only finitely many possibilities for
the vector of Betti numbers. It is easy to check that Corollary \ref{cor 2} gives the restrictions for $\mathbf{b}$ in Table \ref{table2}. (Note
that the odd Betti numbers vanish and that, by Poincar\'e duality, it is only necessary to compute the Betti numbers in Table~\ref{table2}.)

\renewcommand{\arraystretch}{1.5}
\begin{table}[h]
\begin{center}
  \begin{tabular}{ | c || l | l | l | l | }
    \hline
    $k_0$ & $n=2$ & $n=3$ & $n=4$ & $n=5$ \\
    &$ b_2$ & $b_2$ & $(b_2,b_4)$ & $(b_2,b_4)$ \\ \hline \hline
    $1$ & \cellcolor{Apricot}$b_2\leq 4$ & \cellcolor{Apricot}$b_2\leq 7$ &  &  \\ \hline
    $2$ & \cellcolor{SpringGreen}$2$ & $\cellcolor{Apricot}b_2\leq 3$ &  & \\ \hline
    $3$ &  \cellcolor{SpringGreen}$1$ &   \cellcolor{SpringGreen}$1$ & $\cellcolor{Apricot}(1,2),(2,3),(3,1),(4,2),(6,1)$ & \\ \hline
   $4$ &  & \cellcolor{SpringGreen}$1$ & \cellcolor{SpringGreen}$(1,2)$ & \\ \hline
  $5$ & & & \cellcolor{SpringGreen}$(1,1)$ & \cellcolor{Apricot}$(1,1),(6,1)$\\ \hline
   $6$ & & & & \cellcolor{SpringGreen}$(1,1)$\\ 
    \hline
  \end{tabular}
\end{center}
 \caption{List of allowed values of $b_2$ and $b_4$ for $2\leq n\leq 5$.
}
  \label{table2}
\end{table}

\begin{rmk}\label{sharp}
For $n=2$ there are (infinitely many non equivariantly symplectomorphic) examples of monotone Hamiltonian $S^1$-spaces with $b_2(M)=1,2,3,$ and $4$, so the corresponding conditions in Table \ref{table2} are sharp (however $k_0=1$ implies $b_2\geq 2$). They can be obtained from the monotone toric manifolds associated with the polytopes in Figure \ref{fig:smooth}, by restricting the $\T^2$
action to different circle subgroups.
\end{rmk}

An immediate consequence of Corollary \ref{cor 2} is the following:
\begin{corollary}\label{cor 3.1}
Assume that the hypotheses of Corollary \ref{cor 2} hold.
Fixing $n$ and $k_0$, if none of the coefficients of $C(k_0,n,\mathbf{b})$ vanishes, then the Betti numbers with positive coefficients in $C(k_0,n,\mathbf{b})$
determine a finite number of possibilities for the remaining ones.
In particular, if $n\leq 5$, then
\begin{equation}\label{min euler}
k_0=n+1 \implies \chi(M)=n+1\,.
\end{equation}
\end{corollary}

For higher values of $n$, if one also assumes \emph{unimodality} of the vector ${\bf b}=(b_0,\ldots,b_{2n})$ of even Betti numbers, meaning that
$b_{2i}\leq b_{2i+2}$ for all $i\leq \frac{n}{4}$, we can conclude that, for all $k_0\geq n-2$, there are only finitely many possibilities for $\mathbf{b}$.
Note that assuming unimodality of $\mathbf{b}$ is not very restrictive. Indeed, in \cite[Sect.\ 4.2]{Jef}, Tolman asked whether the sequence of
even Betti numbers of a Hamitonian $S^1$-space is unimodal. Since then there has been much work trying to answer this question \cite{Luo,Cho,CK}.
In particular, it is known that unimodality holds whenever the moment map associated with the action is index increasing \cite[Thm.\ 1.2]{Cho}.

\begin{corollary}\label{cor 3}
Let $(M,\omega,\psi)$ be a monotone Hamiltonian $S^1$-space of dimension $2n$ with index $k_0$, and assume it admits a toric 1-skeleton.
Suppose that the vector ${\bf b}=(b_0,\ldots,b_{2n})$ of even Betti numbers is unimodal. Then
\begin{enumerate}
\item If $k_0=n+1$, then $b_{2j}=1$ for all $j=0,\ldots,n$.
\vspace{0.2cm}
\item If $k_0=n$, then
\begin{itemize}
\item[(a)] if $n$ is \emph{odd}, then $b_{2j}=1$ for all $j=0,\ldots,n$;
\item[(b)] if $n$ is \emph{even}, then $b_{2j}=1$ for all $j\in \{0,\ldots,n\}\setminus \{\frac{n}{2}\}$ and $b_n=2$.
\end{itemize}
\vspace{0.2cm}
\item If $k_0=n-1$ and $n\geq 2$, then
 $$b_{2j}=b_{2(n-j)}\leq 2+\left\lfloor \frac{2}{n-1}\right\rfloor \quad\text{for all}\quad  0\leq j \leq \lambda,$$
 where $\lambda=\frac{n-1}{2}-\left\lfloor \frac{n\sqrt{3}-3}{6}\right\rfloor$ if $n$ is \emph{odd}, and $\lambda=\frac{n}{2}-\left\lfloor \frac{n\sqrt{3}}{6}\right\rfloor$ if $n$ is \emph{even}. 
 \\Moreover
there are finitely many possibilities for the other Betti numbers.
\vspace{0.2cm}
\item If $k_0=n-2$ and $n\geq 3$, then

$$b_{2j}=b_{2(n-j)}\leq 4+\left\lfloor \frac{6}{n-1}\right\rfloor \quad\text{for all}\quad  0\leq j \leq \lambda,$$
where $\lambda=\frac{n-1}{2}-\left\lfloor \frac{\sqrt{3n(n-1)}-3}{6}\right\rfloor$ if $n$ is \emph{odd}, and $\lambda= \frac{n}{2}-\left\lfloor \frac{\sqrt{3n(n-1)}}{6}\right\rfloor$ if $n$ is \emph{even}.
Moreover there are finitely many possibilities for the other Betti numbers.

\end{enumerate}

\end{corollary}
\begin{proof} Fixing $n$ and $k_0$ we can see $C(k_0,n,{\bf b})$ as a linear function of ${\bf b}$.
Let $A_i$ be the coefficient of $b_{2i}$ in $C(k_0,n,{\bf b})$ and let $2\lambda$ be the smallest index of the Betti numbers that have a negative coefficient  in $C(k_0,n,{\bf b})$. 
In particular, 
$$
2\lambda=n-2\left\lfloor \sqrt{\frac{n(k_0+1)}{12}} \right\rfloor\text{ if $n$ is even, and }\;2\lambda=n-1-2\left\lfloor -\frac{1}{2}+ \sqrt{\frac{n(k_0+1)}{12}} \right\rfloor\text{ if $n$ is odd}.
$$ 
Moreover, let
$$
S:= \sum_{i=1}^{\lfloor \frac{n}{2} \rfloor} A_i = \frac{1}{2}n(n-1)(n-k_0-3). 
$$
Using the fact that $1=b_0\leq b_2\leq \cdots \leq b_{2\lfloor n/2\rfloor}$, we have from Corollary \ref{cor 2} that
\begin{equation}\label{sum}
0\leq C(k_0,n,{\bf b}) = A_0 + \sum_{i=1}^{\lfloor \frac{n}{2} \rfloor} A_i\, b_{2i} \leq A_0 + S\, b_{2\lambda},
\end{equation}
with $A_0=n(3n - k_0-1)$.
Hence, if $k_0\geq n-2$ we have $S< 0$ and
\begin{equation}\label{bm}
b_{2\lambda}\leq -\frac{A_0}{S}=2\frac{3n-k_0-1}{(n-1)(k_0-n+3)}.
\end{equation}

If $k_0=n+1$ then $b_{2\lambda} \leq 1$ and so, by unimodality of ${\bf b}$, we have $1=b_0=b_2=\ldots=b_{2\lambda}$. We now show by induction that $b_{2j}=1$ for all $\lambda\leq j\leq \lfloor n/2 \rfloor$. Assuming that $b_{2j}=1$, and hence $b_{2i}=1$ for all $i\leq j$, and substituting these values of $b_{2i}$ in \eqref{sum}, we obtain
$$
0\leq C(n+1,n,{\bf b}) = \sum_{i=0}^j A_i + \sum_{i=j+1}^{\lfloor n/2 \rfloor} A_i b_{2i} \leq  \sum_{i=0}^j A_i +  \sum_{i=j+1}^{\lfloor n/2 \rfloor} A_i  b_{2(j+1)},
$$ 
since $A_i<0$ for all $i\geq \lambda$.
Hence,
$$
 b_{2(j+1)} \leq -\frac{ \sum_{i=0}^j A_i }{ \sum_{i=j+1}^{\lfloor n/2 \rfloor} A_i } = \frac{\sum_{i=0}^j A_i }{\sum_{i=0}^j A_i -(S+A_0)} = 1,
$$
implying that $ b_{2(j+1)} =1$. Here we used the fact that, for $k_0=n+1$, we have $S+A_0 = 0$. We conclude that, if $k_0=n+1$, all the even Betti numbers are $1$.

If $k_0=n$ and $n\geq 4$ is even, then by \eqref{bm} we have  $b_{2\lambda} \leq 1$ and so, by unimodality of ${\bf b}$, we have $1=b_0=b_2=\ldots=b_{2\lambda}$. Moreover, assuming  $1=b_{2i}$ for all $i\leq j$ and $\lambda\leq j < n/2 -1 $, substituting these values  in \eqref{sum}, we get
$$
 b_{2(j+1)} \leq 1+\frac{S+A_0}{\sum_{i=0}^j A_i -(S+A_0)} < 1+ \frac{S+A_0}{\sum_{i=0}^{n/2-1} A_i -(S+A_0)} = 2
$$
implying that $ b_{2(j+1)} =1$. Here we used the fact that for $k_0=n$ we have 
$$
\sum_{i=0}^{n/2-1} A_i = 2(S+A_0)= n (n+1)
$$
 and that $j<  n/2  -1$. We conclude that, if $k_0=n$ and $n\geq 4$ is even, all the even Betti numbers up to $b_{n}$ are $1$ and that $b_n$ can be $1$ or $2$. If $n=2$ then $b_0=1$ and $b_2$ can be $1$ or $2$.
 
If $k_0=n$ and $n\geq 3$ is odd then $b_{2\lambda} \leq 1$ and so, by unimodality of ${\bf b}$, we have $1=b_0=b_2=\ldots=b_{2\lambda}$. Moreover, assuming  $b_{2i}=1$ for all $i\leq j$ and $\lambda\leq j \leq (n-1)/2 -1  $, substituting these values  in \eqref{sum}, we get
$$
 b_{2(j+1)} \leq 1+\frac{S+A_0}{\sum_{i=0}^j A_i -(S+A_0)} \leq 1+ \frac{S+A_0}{\sum_{i=0}^{(n-1)/2-1} A_i -(S+A_0)}=  \frac{3}{2} \left(1+ \frac{1}{n^2+n-3} \right) <2
$$
implying that $ b_{2(j+1)} =1$. Here we used the fact that for $k_0=n$ we have 
$$
\sum_{i=0}^{(n-1)/2-1} A_i - (S+A_0)= A_{\frac{n-1}{2}}=n^2 +n-3.
$$
We conclude that, if $k_0=n$ and $n\geq 3$ is odd, all the even Betti numbers  are $1$.

If $k_0=n-1$, then from \eqref{bm} we have 
\begin{equation}\label{bh}
b_{2\lambda}\leq 2+\left\lfloor \frac{2}{n-1}\right\rfloor\,,
\end{equation}
implying that for all $n\geq 4$, 
up to $2\lambda$ all the Betti numbers are $1$ or $2$. Then \eqref{sum} gives
$$
0\leq C(n-1,n,{\bf b}) = A_0 + \sum_{i=1}^{\lambda-1} A_i\, b_{2i} +A_\lambda b_{2\lambda}+ \sum_{i=\lambda+1}^{\left\lfloor\frac{n}{2}\right\rfloor} A_i\, b_{2i} \leq A_0 + 2\sum_{i=1}^{\lambda-1} A_i+A_\lambda+ \sum_{i=\lambda+1}^{\left\lfloor\frac{n}{2}\right\rfloor} A_i\, b_{2i},
$$
implying that
$$
\sum_{i=\lambda+1}^{\left\lfloor\frac{n}{2}\right\rfloor} \lvert A_i\rvert\, b_{2i}\leq A_0 + 2\sum_{i=1}^{\lambda-1} A_i+A_\lambda\,,
$$
which gives a finite number of possibilities for the remaining Betti numbers of $M$. From \eqref{bh} we have that, if $n=2$, then $b_2\leq 4$, and if $n=3$ then $b_2\leq 3$.

Repeating the same procedure for $k_0=n-2$, from \eqref{bm} we have $b_{2\lambda}\leq 4 + \lfloor \frac{6}{n-1}\rfloor $,
and all the claims follow similarly. 
\end{proof}

\subsection{Delzant reflexive polytopes: indices, $f$-vectors and $h$-vectors}

\begin{defin}\label{index2}
Let $\Delta$ be a Delzant reflexive polytope of dimension $n$. The {\bf index} $k_0$ of $\Delta$ is defined as
$$
k_0:=\operatorname{gcd}\{l(e)\mid e\in \Delta[1]\}\,.
$$
\end{defin}
From Proposition \ref{index gcd} it is clear that the index of $\Delta$ is the same as the index of the corresponding symplectic toric manifold $(M_\Delta,\omega,\psi)$.
Hence it agrees with the standard notion of index defined in algebraic geometry, and satisfies $1\leq k_0 \leq n+1$ (see \eqref{bound k0}).
Moreover, it can be seen that $k_0$ is the largest integer $l$ such that $\frac{1}{l}\Delta$ is a (Gorenstein) integral polytope (compare with \cite[Sect. 4.2]{Sa}).

Just like in the previous subsection, there is a clear relation between the index and the $f$-vector, or the $h$-vector of $\Delta$.
The following corollary is an immediate consequence of Theorem \ref{main combinatorics}.

\begin{corollary}\label{inequalities h f}
Let $\Delta$ be a Delzant reflexive polytope of dimension $n$ with $f$-vector $\mathbf{f}$ and $h$-vector $\mathbf{h}$. 
Consider the integers 
$$
C(k_0,n,\mathbf{f}):=12f_2+(5-3n-k_0)f_1
$$
and $C(k_0,n,\mathbf{h})$ defined, for $n=2m$ even, as
$$
\displaystyle\sum_{k=1}^{m}\Big[12k^2-2m(k_0+1)\Big]h_{m-k}- m(k_0+1)\,h_m, 
$$
and, for $n=2m+1$ odd, as
$$
\displaystyle\sum_{k=1}^{m}\Big[12k(k+1)+3-(2m+1)(k_0+1)\Big]h_{m-k}- \big[(2m+1)(k_0+1)-3\big]h_{m}. 
$$
Then $C(k_0,n,\mathbf{f})=C(k_0,n,\mathbf{h})$ is a \emph{non-negative} multiple of $k_0$. 
Moreover, these numbers are zero if and only if $l(e)=k_0$ for all $e\in \Delta[1]$.
\end{corollary}
\begin{rmk}
$\;$\vspace{0.1cm}
\begin{enumerate}
\item For $n=2$, Corollary \ref{inequalities h f} gives $h_1\leq 4$, hence $f_0\leq 6$. The inequality is sharp, and is attained by the reflexive hexagon (see Figure \ref{fig:smooth}).
However, for $n=3$, one obtains $h_1\leq 7$, hence $f_0\leq 16$. This inequality is not sharp, since from the classification of reflexive polytopes of dimension $3$, one has $f_0\leq 14$.
\item Delzant reflexive polytopes are dual to Fano polyhedra. In \cite[Theorem 2.3.7]{Bat1} Batyrev proves an inequality for simplicial Fano polyhedra that is equivalent to 
$C(1,n,\mathbf{f})\geq 0$.
\end{enumerate}
\end{rmk}

The next result is the combinatorial analogue of Corollary \ref{cor 3}. Observe that the unimodality of the $h$-vector holds for every simple polytope.
This was proved by Stanley in the celebrated paper \cite{St}. 
Note that much more is known on (toric) Fano varieties. Indeed, for $k_0\geq n-2$ they are completely classified \cite{Ar,W}.
However, the proof of Corollary \ref{cor k0} relies on different tools, not involving any of the techniques used in algebraic geometry. 
Instead, it is a very special case of Corollary \ref{cor 3}, which applies to a class of monotone Hamiltonian $S^1$-spaces  which is much broader than the class of toric Fano manifolds, and is not classified for $n>2$. 

\begin{corollary}\label{cor k0}
Let $\Delta$ be a Delzant reflexive polytope of dimension $n$ and $h$-vector $\mathbf{h}=(h_0,\ldots,h_n)$. Let $k_0$ be the index of $\Delta$.
Then
\begin{enumerate}
\item If $k_0=n+1$, then $\Delta$ is a Delzant reflexive simplex. 
\vspace{0.2cm}
\item If $k_0=n$, then $n=2$, and $\Delta$ is $GL(2,\Z)$-equivalent to the reflexive square.
\vspace{0.2cm}
\item If $k_0=n-1$ and $n\geq 2$, then
 $$h_{j}=h_{n-j}\leq 2+\left\lfloor \frac{2}{n-1}\right\rfloor \quad\text{for all}\quad  0\leq j \leq \lambda,$$
 where $\lambda=\frac{n-1}{2}-\left\lfloor \frac{n\sqrt{3}-3}{6}\right\rfloor$ if $n$ is \emph{odd}, and $\lambda=\frac{n}{2}-\left\lfloor \frac{n\sqrt{3}}{6}\right\rfloor$ if $n$ is \emph{even}. 
 \\Moreover,
there are finitely many possibilities for the remaining $h$-numbers.
\vspace{0.2cm}
\item If $k_0=n-2$ and $n\geq 3$, then
$$h_{j}=h_{n-j}\leq 4+\left\lfloor \frac{6}{n-1}\right\rfloor \quad\text{for all}\quad  0\leq j \leq \lambda,$$
where $\lambda=\frac{n-1}{2}-\left\lfloor \frac{\sqrt{3n(n-1)}-3}{6}\right\rfloor$ if $n$ is \emph{odd}, and $\lambda= \frac{n}{2}-\left\lfloor \frac{\sqrt{3n(n-1)}}{6}\right\rfloor$ if $n$ is \emph{even}.
Moreover, there are finitely many possibilities for the remaining $h$-numbers.

\end{enumerate}
\end{corollary}
\begin{proof}
The proof follows exactly that of Corollary \ref{cor 3}, by replacing Corollary \ref{cor 2} with Corollary \ref{inequalities h f} and the vector of Betti numbers $\mathbf{b}$
with $\mathbf{h}$, which is automatically unimodal. 
The only thing which must be proved is the claim in (2). By the same procedure as in the proof of Corollary \ref{cor 3} (2), we have that, if $n$ is odd, then $h_i=1$ for all 
$i=0,\ldots,n$, and, if $n$ is even, then $h_i=1$ for all
$i\neq \frac{n}{2}$, and $h_{\frac{n}{2}}=2$.  

If $n$ is odd, then the Delzant polytope $\Delta$ has $n+1$ vertices, implying that it is a simplex. 
Since it is Delzant, the relative length of all its edges must be the same, because a Delzant simplex is, modulo rescaling, $GL(n;\Z)$-equivalent to the 
standard simplex $$\Delta_0=\{(x_1,\ldots,x_n)\mid  x_1+\cdots+x_n\leq 1,\;\; x_i\geq 0, \,\, i=1,\ldots,n\}\,.$$
Since $k_0=n$, it follows that $\Delta$ is $GL(n;\Z)$-equivalent to $n\Delta_0$. However, from Theorem \ref{main combinatorics}, we have that
$$n\lvert \Delta[1]\rvert=\sum_{e\in \Delta[1]}l(e)=\frac{1}{2}n(n+1)^2\,,$$
with $\lvert \Delta[1]\rvert=\frac{1}{2}n(n+1)$, which is impossible.

If $n$ is even, then $\Delta$ is a Delzant polytope with $n+2$ vertices. From the classification of polytopes with $n+2$ vertices \cite[Sect.\ 6.1]{Gr} we see that
the only cases where the polytope is simple occur in dimension $2$. From the classification of Delzant reflexive polygons, the only possible case with $k_0=2$ is the reflexive square, modulo $GL(2;\Z)$-transformations (see Figure \ref{fig:smooth}).
\end{proof}

\subsection{Reflexive GKM graphs}\label{mg}

Theorem \ref{A} allows us to generalize Theorem \ref{main combinatorics} to a larger class of objects, called \emph{reflexive (GKM) graphs}.

Let $(M,\omega,\psi)$ be a Hamiltonian GKM space, with an effective action of a torus $\T$ of dimension $d$.
As we pointed out in Remark \ref{embed gkm}, using the moment map $\psi$, the GKM graph $(V,E_{GKM})$ associated to $(M,\omega,\psi)$ can be regarded as a graph in $Lie(\T)^*\simeq\R^d$. It is indeed possible to prove that such graph cannot be contained in any affine subspace of $\R^d$, as this would contradict effectiveness of the action. In contrast with the graph associated to a symplectic toric manifold (the 1-skeleton of the corresponding polytope), this graph may not be embedded in $\R^d$ (see Figure \ref{monotone graphs}, where
different edges intersect in points which are not vertices). In the following, we require that the moment map is \emph{injective on the fixed point set}. With abuse of notation, we also denote
the image of the GKM graph in $\R^d$ by $(V,E_{GKM})$, and call it a GKM graph. With this convention we have that vertices $v\in V$ are points in $\R^d$ (which we mark with a blue dot) and edges $e\in E_{GKM}$ are segments in $\R^d$ between vertices (see Figure \ref{monotone graphs}).

Note that the number of edges incident to a vertex is always the same. This follows from the injectivity of $\psi$ on the fixed point set, the fact that,
for each edge $e=(v_1,v_2)$, the segment
$\psi(v_2)-\psi(v_1)$ has the same direction as the weight $w_e$ associated to the edge (Lemma \ref{volume}), and for each vertex $v$, the weights 
of the edges $(v,v_i)$ are pairwise linearly independent. Hence the graph $(V,E_{GKM})$ is \emph{regular}, with degree equal to $n=\dim(M)/2$. 
Observe that, in contrast with the toric case, $n$ is not necessarily equal to $d=\dim(\T)$. Since the orbits of $\T$ are isotropic and the action is 
effective, we must have $n\geq d$. In literature, the difference $n-d$ is called the \emph{complexity} of the Hamiltonian (GKM) $\T$-space.

Now we can define the objects that replace the concept of reflexive (Delzant) polytope.
First of all, denote by $\ell^*\subset \R^d$ the dual lattice of the torus $\T$. Here we do not necessarily identify it with $\Z^d$.
\begin{defin}[\textbf{Reflexive GKM graphs}]\label{def monotone}
Let $(M,\omega,\psi)$ be a GKM space with $\psi$ injective on the fixed point set, and $(V,E_{GKM})$ the corresponding GKM graph in $\R^d$. 
Such graph is called \emph{reflexive} if, for every vertex $v\in V\subset \R^d$, the following condition holds:
\begin{equation}\label{monotone vertex}
\sum_{j=1}^n w_j = -v\,,
\end{equation}
where $w_1,\dots,w_n$ are the primitive vectors in $\ell^*$ pointing along the edges of $E_{GKM}$ incident to $v$.
\end{defin}
Note that, \emph{mutatis mutandis}, \eqref{monotone vertex} is exactly \eqref{vertex fano} in Proposition \ref{equivalent}. 
Indeed, as already remarked, symplectic toric manifolds are special cases of Hamiltonian GKM spaces, and the graph corresponding to
the 1-skeleton of a Delzant reflexive polytope $\Delta$ is a reflexive GKM graph.

In analogy with Gorenstein polytopes as generalizations of reflexive polytopes, it is possible to define {\bf Gorenstein (or monotone) GKM graphs} as those
GKM graphs $\Gamma_{GKM}$ associated to a GKM space $(M,\omega,\psi)$ as above, such that there exists $r>0$ for which
\begin{equation}\label{gor gkm}
\sum_{j=1}^n w_j = -r v\,,
\end{equation}
for every vertex $v\in V$, where $w_1,\dots,w_n$ are the primitive vectors in $\ell^*$ pointing along the edges of $E_{GKM}$ incident to $v$.
In analogy with Gorenstein polytopes, we can refer to $r$ as the {\bf index} of $\Gamma_{GKM}$.
Note that the index $r$ should not be confused with the index $k_0$ of the underlying Hamiltonian (GKM) space, see Definition \ref{psindex}.
 However, it can be proved that if $\Gamma_{GKM}$ is a Gorenstein GKM graph which is `primitive', namely $\gcd\{l(e)\mid e\in E_{GKM}\}=1$,
 and all of its vertices are in the lattice $\ell^*$,
 then the index $r$ coincides with the index $k_0$ of the corresponding Hamiltonian GKM space.
The graphs in Figure \ref{flags pic} are examples of (primitive) Gorenstein GKM graphs with index $r$ given respectively by $2$ and $3$.

Figures~\ref{monotone graphs} and \ref{A2} give  examples of reflexive graphs, all associated to flag varieties. In Proposition \ref{graph flag mon} we 
exhibit a whole class of reflexive graphs associated to
 flag varieties.
 \begin{rmk}\label{origin}
 $\;$\vspace{0.1cm}
 \begin{itemize}
 \item[(1)] From Definition \ref{def monotone} it follows that the vertices of a reflexive graph are all elements of $\ell^*$.
 \item[(2)] We recall that reflexive polytopes have a unique interior lattice point, which is set to be the origin. Reflexive graphs may not have a unique
 interior point in $\ell^*$, as Figure \ref{monotone graphs} shows. However, condition \eqref{monotone vertex} implies that the sum of all their vertices 
 is zero, implying that the origin is an interior point of the convex hull of all the vertices $Conv(V)$. Indeed, the sum over all the vertices of the left hand side of \eqref{monotone vertex}  always contains pairs of the form $w_i$ and $-w_i$. Hence the origin is still a
  `special' interior point of $Conv(V)$. 
 \end{itemize}
\end{rmk}
\begin{prop}\label{equivalent monotone}
Let $(M,\omega,\psi)$ be a GKM space with $\psi$ injective on the fixed point set, and $(V,E_{GKM})$ the corresponding GKM graph in $\R^d$. 
Then the following are equivalent:
\begin{itemize}
\item[(I)] $(V,E_{GKM})$ is reflexive;
\item[(II)] $(M,\omega,\psi)$ is monotone, with $c_1=[\omega]$.
\end{itemize}
\end{prop}
Since symplectic toric manifolds are special GKM spaces, Proposition \ref{equivalent monotone} is a generalization of Proposition \ref{equivalent2}.
Moreover, it can be immediately generalized, to say that having a Gorenstein GKM graph of index $r$ is equivalent to having a
monotone GKM space $(M,\omega,\psi)$ with $c_1=r[\omega]$. (Note that, from Lemma \ref{lambda positive}, this constant $r$ is necessarily positive.)
\begin{proof}
First of all observe that, by Definition~\ref{def:GKMgraph} (c') and Remark \ref{embed gkm} (c''), the vectors $w_1,\ldots,w_n$ in \eqref{monotone vertex} are exactly the weights of the isotropy
representation of $\T$ at $p\in M$, where $p$ is the unique fixed point such that $\psi(p)=v$. Hence \eqref{monotone vertex} is equivalent to \eqref{chern condition 2}.
Then the proof of Proposition \ref{equivalent monotone} is \emph{verbatim} the proof of Proposition \ref{equivalent2}, where we do not use the fact that the action is toric,  but just that it is Hamiltonian  with isolated fixed points.
\end{proof}
Before stating the analogue of Theorem \ref{main combinatorics}, we need to define the analogue of the $h$-vector for a GKM graph.
Inspired by Lemma \ref{hequiv},
let $\xi\in \R^d$ be a generic vector, hence $\langle w_e, \xi \rangle \neq 0$ for all $e\in E_{GKM}$, and use it to direct the edges in $E_{GKM}$.
For every generic vector $\xi\in \R^d$, define the $h^{\xi}$-vector of $(V,E_{GKM})$ to be $\mathbf{h}^\xi=(h_0^\xi,\ldots,h_n^\xi)$, where 
\begin{equation}\label{hxi}
h_j^\xi:=\{\# \mbox{ of vertices with }j \mbox{ entering edges}\}\quad\mbox{for all}\quad j=0,\ldots,n,
\end{equation}
and $2n$ is the dimension of the corresponding GKM manifold $M$.
The next result is the analogue of Lemma \ref{hequiv}.
\begin{lemma}\label{hequiv gkm}
Let $(M,\omega,\psi)$ be a GKM space of dimension $2n$, with $\psi$ injective on the fixed point set, and let $(V,E_{GKM})$ be the corresponding GKM graph in $\R^d$. 
Then the following two vectors associated to the GKM space above are the same:
\begin{itemize}
\item[(1)] $\mathbf{h}^\xi=(h_0^\xi,\ldots,h_n^\xi)$ for a generic $\xi\in \R^d$
\item[(2)] $\mathbf{b}=(b_0,b_2,\ldots,b_{2n})$, the vector of even Betti numbers of $M$.
\end{itemize}
In particular $\mathbf{h}^{\xi}$ is independent of the generic $\xi$ chosen.
\end{lemma}
\begin{proof}
The proof of equivalence between $\mathbf{h}^\xi$ and $\mathbf{b}$ is identical to that of Lemma \ref{hequiv}.
\end{proof}
\begin{rmk}
The independence of $\mathbf{h}^\xi$ on the generic $\xi\in \R^d$ chosen can be proved entirely combinatorially \cite[Thm.\ 1.3.1]{GZ}.
\end{rmk}
\begin{defin}\label{gkm h vector}
Given a GKM graph $(V,E_{GKM})$ of degree $n$, the $h$-vector $\mathbf{h}=(h_0,\ldots,h_n)$ is defined to be the $h^\xi$-vector $\mathbf{h}^\xi=(h_0^\xi,\ldots,h_n^\xi)$, for some generic $\xi\in \R^d$.
\end{defin}
\begin{exm}
For the GKM graphs in Figure \ref{monotone graphs}, the $h$-vectors are respectively $(1,2,2,2,1)$, $(1,1,1,1)$ and $(1,1,1,1)$.
\end{exm}
The next corollary is a generalization of Theorem \ref{main combinatorics}.
\begin{corollary}\label{main combinatorics 2}
Let $(V,E_{GKM})$ be a reflexive GKM graph associated to the Hamiltonian GKM space $(M,\omega,\psi)$.
Let $l(e)$ be the relative length of $e$, for every $e\in E_{GKM}$, and
$\mathbf{h}=(h_0,\ldots,h_n)$ the $h$-vector of $(V,E_{GKM})$. 
Then $\displaystyle\sum_{e\in E_{GKM}}l(e)$ only depends on $\mathbf{h}$. 
More precisely,
\begin{equation}\label{length gkm}
\sum_{e\in E_{GKM}}l(e)=C(n,\mathbf{h}),
\end{equation}
where $C(n,\mathbf{h})$ is the integer defined in \eqref{def C(n,h)}.

\end{corollary}
\begin{proof}
The proof is entirely analogous to the symplectic proof of Theorem \ref{main combinatorics} on page \pageref{proof symplectic}.
Indeed, Proposition \ref{equivalent monotone} implies that
$(M,\omega,\psi)$ is monotone with $c_1=[\omega]$. Hence $c_1[S^2_e]=\int_{S^2_e}i^*\omega=l(e)$ for every $e\in E_{GKM}$, where the last equality follows from Lemma \ref{volume}. 
Now the claim follows from Theorem \ref{A}, Lemma \ref{hequiv gkm} and the fact that $\chi(M)=\lvert V\rvert=\sum_{j=0}^n h_j$.
\end{proof}

\begin{rmk}\label{ggkm}
The above corollary can be immediately generalized to Gorenstein GKM graphs of index $r>0$, for which one has
$$
\sum_{e\in E_{GKM}}l(e)=\frac{1}{r}C(n,\mathbf{h})
$$
(see also Remark \ref{equiv 3} (2)).
\end{rmk}

\subsubsection{Examples of reflexive GKM graphs: Coadjoint orbits}\label{flags}

In this subsection we exhibit an interesting class of reflexive GKM graphs, namely those arising as the GKM graphs of coadjoint orbits of compact simple Lie groups (or equivalently flag varieties)
endowed with a monotone symplectic structure, with $c_1=[\omega]$.
In the following we recall facts about such manifolds regarded as GKM spaces. More details can be found in \cite{GHZ}, \cite[Sect.\  4.2]{GSZ} 
and \cite[Sect.\ 6]{ST}.

Let $G$ be a compact simple Lie group with Lie algebra $\mathfrak{g}$, and let $\T\subset G$
be a maximal torus with Lie algebra $\mathfrak{t}$.
Let $\langle \cdot,\cdot \rangle$ be a positive definite, symmetric, $G$-invariant bilinear form on $\mathfrak{g}$, and use it to view $\mathfrak{t}^*$ as a subspace of $\mathfrak{g}^*$.
Denote by $R\subset \mathfrak{t}^*$ the set of roots, by $R^+$ a choice of positive roots, and by $R_0$ the corresponding simple roots. 
Let $W$ be the Weyl group of $G$, which is generated by the reflections $s_\alpha\colon \mathfrak{t}^*\to \mathfrak{t}^*$ across the hyperplanes
$H_\alpha$ orthogonal to the simple roots $\alpha\in R_0$. We denote the action of an element $w\in W$ on $\beta\in \mathfrak{t}^*$ by $w(\beta)$.
We recall that for every $\alpha\in R$ and $w\in W$ one has
\begin{equation}\label{swap}
s_{w(\alpha)}w=w\,s_\alpha\,.
\end{equation}

For any choice of a subset $\emptyset\subseteq I\subset R_0$, let $W_I$ be the subgroup of $W$ generated by reflections $s_\alpha$ with $\alpha\in I$,
and let $\langle I \rangle \subseteq R^+$ be the set of positive roots that can be written as linear combinations of elements of $I$.
Notice that any reflection in $W_I$ is of the form $s_\alpha$, for some $\alpha\in \langle I \rangle$ (see \cite[Sect.\ 1.14]{Hum}). 
Moreover, the set $R^+\setminus \langle I \rangle$ is $W_I$-invariant \cite[Lemma 4.1]{GSZ}. 

For any such $I$, we define the following abstract graph $\Gamma_I=(V_I,E_I)$:
\begin{itemize}
\item[$\bullet$] The vertices are the right cosets
$$
W/W_I=\{w W_I\mid w\in W\}=\{[w]\mid w\in W\}\,.
$$
\item[$\bullet$] Two vertices $[v]$ and $[w]$ are joined by an edge if and only if $[v]=[w s_\alpha]$ for some $\alpha\in R^+\setminus \langle I \rangle$.
The invariance of $R^+\setminus \langle I \rangle$ under the action of $W_I$ proves that this definition makes sense.
Note that, for every $[v],\,[w]\in W/W_I$, there exists an edge from $[v]$ to $[w]$ exactly if there is one from $[w]$ to $[v]$. Instead of
having such pair of directed edges, we just take one undirected edge with endpoints $[ws_\alpha]=[v]$ and $[w]$. Thus this graph has \emph{unoriented} edges.
\end{itemize}
We define a map from $\Gamma_I$ to $\mathfrak{t}^* \simeq \R^d$ that restricts to a bijection on $V_I$, such that the image of $\Gamma_I$ in $\R^d$
is the GKM graph of a Hamiltonian $\T$-space. 

Let $p_0\in \mathfrak{t}^*$ be a point lying in the intersection of the hyperplanes $\bigcap_{\alpha\in I}H_\alpha$.
The point $p_0$ is said to be \emph{generic} in this intersection if
$s_\alpha(p_0)\neq p_0$ for all $\alpha \in R^+\setminus \langle I\rangle$. Since $p_0$ is not orthogonal to any of the roots in $R^+\setminus \langle I \rangle$, we can assume that
$\langle p_0, \alpha \rangle<0$ for all $\alpha\in R^+\setminus \langle I \rangle$.

For any such choice, define the following map:
\begin{equation}\label{map flag}
\Psi_{p_0}\colon W/W_I \to \mathfrak{t}^*\,, \quad \Psi_{p_0}([v])=v(p_0)\,.
\end{equation}
Since $W_I$ acts trivially on $\bigcap_{\alpha\in I}H_\alpha$, and since $p_0$ belongs to this intersection, the map $\Psi_{p_0}$ is well-defined. 
\begin{lemma}\label{psi inj}
The map $\Psi_{p_0}$ defined in \eqref{map flag} is injective. 
\end{lemma}
\begin{proof}
Injectivity is equivalent to saying that the stabilizer group $W_{p_0}$ of $p_0$ is exactly $W_I$. But this follows from \cite[Thm.\ 1.12 (c)]{Hum}:
$W_{p_0}$ is generated by the reflections that it contains. Since $W_{p_0}$ contains all reflections $s_\alpha$ for $\alpha\in \langle I\rangle$, but no
$s_\alpha$, for $\alpha\in R^+\setminus \langle I\rangle$, the conclusion follows. 
\end{proof}
With the map \eqref{map flag} at hand, we can define the graph $\Gamma_I(p_0)=(V,E)$ as the `image' of the abstract graph $\Gamma_I$ in $\R^d$, where $\R^d$ is endowed with lattice $\ell^*=\Z\langle \alpha_1,\ldots,\alpha_d\rangle$, and $R_0=\{\alpha_1,\ldots,\alpha_d\}$. Hence we have that:
\begin{itemize}
\item[$\bullet$] The vertex set $V$ is obtained by \eqref{map flag};
\item[$\bullet$] Two vertices $v(p_0)$ and $w(p_0)$ are joined by an edge if and only if $$v(p_0)-w(p_0)= m\cdot w(\alpha)$$ for some $\alpha\in R^+\setminus \langle I \rangle$ and $m\in \R\setminus \{0\}$.
Indeed, from \eqref{swap} we have that if $[v]=[w s_\alpha]=[s_{w(\alpha)}w]$, then $$v(p_0)-w(p_0)=s_{w(\alpha)}w(p_0)-w(p_0)=m \cdot w(\alpha)\,.$$
The converse follows similarly.
\end{itemize}
Now we introduce the Hamiltonian GKM space which has $\Gamma_I(p_0)$ as the corresponding GKM graph. 
Given $p_0\in \mathfrak{t}^*$, consider its coadjoint orbit $\mathcal{O}_{p_0}:=G \cdot p_0\subset \mathfrak{g}^*$. Then $\mathcal{O}_{p_0}$ can be endowed with the Kostant--Kirillov symplectic form $\omega$, and the induced action of $\T$ on $(\mathcal{O}_{p_0},\omega)$ is Hamiltonian with moment map given by the inclusion $\mathcal{O}_{p_0}\hookrightarrow \mathfrak{g}^*$ followed by the projection
 $\mathfrak{g}^*\to \mathfrak{t}^*$.  As it is carefully described in \cite{GHZ}, this action is GKM with GKM graph given exactly by $\Gamma_I(p_0)$. 
 
 The GKM condition, which translates into proving that the isotropy weights at each $\T$-fixed point are pairwise linearly independent, is easy to check. 
 Let $\mathcal{W}(p)$ be the set of such weights.
 Consider first the right coset $W_I/W_I$ which, under the map $\Psi_{p_0}$, is mapped exactly to $p_0$. First we compute $\mathcal{W}(p_0)$. 
 Note that $p_0$ is connected, via the GKM graph, to all points of the form $s_\alpha(p_0)$, for all $\alpha\in R^+\setminus \langle I \rangle$. Moreover, the weight
 of the isotropy representation of $\T$ on the tangent space at $p_0$ of the sphere with fixed points $p_0$ and $s_{\alpha}(p_0)$, 
 is the unique primitive integral vector in $\ell^*$ parallel to $\psi(s_{\alpha}(p_0))-\psi(p_0)=s_{\alpha}(p_0)-p_0$ with the same direction (see (c'') and Lemma \ref{volume}).
 Since $s_\alpha(p_0)-p_0=-2\frac{\langle p_0, \alpha \rangle}{\langle \alpha, \alpha \rangle}\alpha$, and by choice $\langle p_0, \alpha \rangle <0$, it follows that this weight is exactly $\alpha$, and so
 \begin{equation}\label{weights p0}
 \mathcal{W}(p_0)= R^+ \setminus \langle I \rangle\,.
 \end{equation}
In analogy with the above argument, the vertex $w(p_0)$ is joined in the GKM graph to vertices of the form $ws_{\alpha}(p_0)=s_{w(\alpha)}w(p_0)$. 
By $G$-invariance of $\langle \cdot , \cdot \rangle$ we have that $$s_{w(\alpha)}w(p_0)-w(p_0)=-2\frac{\langle w(p_0), w(\alpha) \rangle}{\langle w(\alpha), w(\alpha) \rangle}w(\alpha)=-2\frac{\langle p_0, \alpha \rangle}{\langle \alpha, \alpha \rangle}w(\alpha)\,,$$ implying that
\begin{equation}\label{weights w}
\mathcal{W}(w(p_0)) = w(R^+ \setminus \langle I \rangle)=\{w(\alpha)\mid \alpha \in R^+\setminus \langle I \rangle\}.
\end{equation}  
In the next proposition we find the point $p_0\in \mathfrak{t}^*$ such that the corresponding coadjoint orbit is monotone.  
\begin{prop}\label{graph flag mon}
For $p_0=-\sum_{\alpha\in R^+\setminus \langle I \rangle}\alpha$ the GKM graph $\Gamma_I(p_0)$ is reflexive, and the
coadjoint orbit $(\mathcal{O}_{p_0},\omega)$ is monotone, with $c_1=[\omega]$.
\end{prop}
\begin{proof}
First of all, we prove that $p_0$ lives in $\cap_{\alpha\in I}H_\alpha$, and that $s_\alpha(p_0)\neq p_0$ for all $\alpha\in R^+\setminus \langle I \rangle$. 
By \cite[Sect.\ 1.15]{Hum}, it is enough to prove that the isotropy group of $p_0$ under the action of $W$ is exactly $W_I$. 
Observe that the isotropy group of $p_0$ contains $W_I$ since, as already remarked, $R^+\setminus \langle I \rangle$ is $W_I$-invariant, hence the sum of its elements is fixed
by $W_I$.  

Now suppose that $w\in W$ satisfies $w(p_0)=p_0$. This implies that $w$ leaves the set $R^+\setminus \langle I \rangle$ invariant. Indeed,
write $-p_0$ as a sum of simple roots, and let $n_i$ be the coefficients in this sum of the simple roots $\alpha_i\in R_0\setminus I$. 
By the choice of $p_0$, each $n_i$ is strictly positive. If $w$ did not leave the set $R^+\setminus \langle I \rangle$ invariant, then one of its positive roots would be sent to a root in $\langle I \rangle$.
However, if we write each of the roots in $\langle I \rangle$ as a combination of simple roots, by definition their coefficients w.r.t.\ the simple roots in $R_0\setminus I$ are zero. 
Hence one of the $n_i$'s would decrease, contradicting $w(p_0)=p_0$.
So to conclude that the isotropy group of $p_0$ is contained in $W_I$, it is enough to prove that, if $w$ leaves the set $R^+\setminus \langle I \rangle$ invariant, then $w\in W_I$.

The proof of this fact is by induction on $l(w)$. If $l(w)=1$, then $w$ must be in $W_I$. In fact, if $w=s_\alpha$ for some $\alpha\in R_0\setminus  I $,
then $s_\alpha$ would not leave $R^+\setminus \langle I \rangle$ invariant, as $s_{\alpha}(\alpha)=-\alpha$. If $l(w)>1$, let $R_w^+:=\{\alpha \in R^+ \mid w(\alpha)\in -R^+\}$.
Since by hypothesis $w$ leaves $R^+\setminus \langle I \rangle$ invariant, $R_w^+$ must be contained in $\langle I \rangle$. Observe that $R_w^+$ contains a simple root. Indeed, write
$w$ in reduced expression as $s_{i_1}\cdots s_{i_k}$, where $\alpha_{i_l}$ is a simple root and $s_{i_l}$ denotes the corresponding reflection, for all $l=1,\ldots,k$. Then by \cite[Lemma 1.3.14]{Kumar} we have that $\alpha_{i_k}\in R^+_w$,
and so $\alpha_{i_k}\in I$ and $s_{i_k}\in W_I$. Thus $w=w' s_{i_k}$, where $w':=w s_{i_k}$ also leaves $R^+\setminus \langle I \rangle$ invariant, and $l(w')=l(w)-1$. 

$\;$\\

Now we need to check that $\Gamma_I(p_0)$ is a reflexive GKM graph. In order to do so, it is sufficient to observe that, by definition of $p_0$, we have $\sum_{\alpha\in R^+\setminus \langle I \rangle}w(\alpha)=-w(p_0)$, and that the sum on the left hand side is, by \eqref{weights w}, precisely the sum of the weights at $w(p_0)$.
By Definition \ref{def monotone}, the GKM graph $\Gamma_I(p_0)$ is reflexive, and by Proposition \ref{equivalent monotone}, $(\mathcal{O}_{p_0},\omega)$ is a monotone symplectic manifold with $c_1=[\omega]$. 
\end{proof}

\begin{figure}[h]
\begin{center}
\includegraphics[width=8cm]{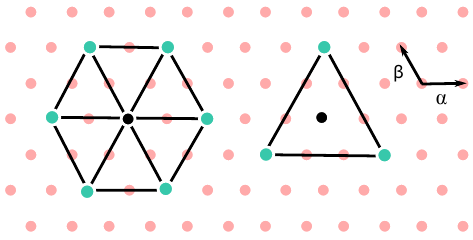}
\caption{Examples of reflexive GKM graphs associated to coadjoint orbits of type $A_2$.}
\label{A2}
\end{center}
\end{figure}

\begin{exm}\label{a2b2}
We give examples of reflexive GKM graphs arising from the coadjoint orbits of $SU(3)$, $SU(4)$ and $SO(5)$.

For $SU(3)$, the root system is of type $A_2$, with simple roots $R_0=\{\alpha,\beta\}$ (see Figure \ref{A2}).
The graph on the left corresponds to the choice of $I=\emptyset$, hence $p_0=-2\alpha-2\beta$. The one on the right 
to $I=\{\beta\}=\langle I \rangle$, hence $p_0=-2\alpha-\beta$.

For $SU(4)$, the root system is of type $A_3$, with simple roots $R_0=\{\alpha,\beta,\gamma\}$.
Let $I=\{\alpha,\gamma\}$ (the two simple roots at the extreme points of the Dynkin diagram), and $\langle I \rangle = I$. 
Then $p_0=-2(\alpha+2\beta+\gamma)$. The reader can check that the associated (reflexive) GKM graph is the 1-skeleton of an octahedron.

For $SO(5)$, the root system is of type $B_2$, with simple roots $R_0=\{x_1-x_2,x_2\}$. Here we identify the dual of the Lie algebra of a maximal compact
torus of $SO(5)$ with $\R^2$ and lattice $\Z^2$, and consider its standard $\Z$-basis $x_1$, $x_2$. In this case there are three reflexive GKM graphs (see Figure \ref{monotone graphs}).
The graph on the left corresponds to the choice of $I=\emptyset$, hence $p_0=-3x_1-x_2$. The one in the middle corresponds to $I=\{x_1-x_2\}=\langle I \rangle$, and $p_0=-2x_1-2x_2$.
The graph on the right to $I=\{x_2\}=\langle I \rangle$, and $p_0=-3x_1$.
\end{exm}

\begin{rmk}\label{octa}
The coadjoint orbit of $SU(4)$ that we describe above corresponds to a Grassmannian of complex planes in $\C^4$. In the associated reflexive GKM graph $\Gamma_{GKM}=(V,E_{GKM})$,
the relative length of all of its edges is $4$. Since the $h$-vector of $\Gamma_{GKM}$ is given by $(1,1,2,1,1)$, Corollary \ref{main combinatorics 2} gives 
$$
\sum_{e\in E_{GKM}}l(e)= 44 + 8 h_1-4h_2=48=4\cdot 12\,.
$$
Let $\Delta$ be the reflexive octahedron in $\R^3$ (the dual to the reflexive cube). Since it is not simple, we cannot apply Theorem \ref{main combinatorics} to $\Delta$.
However, it is interesting to notice that the 1-skeleton of $\Delta$ is a `Gorenstein graph' of index 4. For instance, let $\{x_1,x_2,x_3\}$ be the standard basis of $\R^3$. Then
at the vertex $x_1$, equation \eqref{gor gkm} holds. Indeed $$(-x_1+x_3)+(-x_1-x_3)+(-x_1+x_2)+(-x_1-x_2)=-4 x_1\,.$$ Hence we can apply the formula
in Remark \ref{ggkm} to count the sum of the relative lengths of its edges, which is exactly $12$.

\end{rmk}

\end{document}